\documentclass[a4paper,12pt,draft]{amsart}

\usepackage[margin=25mm]{geometry}
\usepackage{amssymb, amsmath, amsthm, amscd}
\usepackage{eucal,mathrsfs,dsfont}
\usepackage{color}

\renewcommand{\le}{\leqslant}
\renewcommand{\ge}{\geqslant}

\renewcommand{\d}{~\text{d}}

\definecolor{mno}{rgb}{0.5,0.1,0.5}

\newcommand{\R}{\mathds R}
\newcommand{\Q}{\mathds Q}
\newcommand{\Ss}{\mathds S}

\newcommand{\Pp}{{\mathds P}}
\newcommand{\Ee}{{\mathds E}}

\newcommand{\I}{\mathds 1}
\newcommand{\N}{\mathds{N}}
\newcommand{\Z}{\mathds Z}

\newcommand{\w}{{\omega} }

\newtheorem{theorem}{Theorem}[section]
\newtheorem{lemma}[theorem]{Lemma}
\newtheorem{proposition}[theorem]{Proposition}

\numberwithin{equation}{section}

\theoremstyle{definition}

\newtheorem{remark}[theorem]{Remark}

\begin{document}
\allowdisplaybreaks
\title[] {\bfseries
Quenched Asymptotics for Symmetric L\'evy Processes
interacting with Poissonian fields}
\author{Jian Wang}
\thanks{\emph{J.\ Wang:} College of Mathematics and Informatics  \&
Fujian Key Laboratory of Mathematical Analysis and Applications (FJKLMAA)  \& Center for Applied Mathematics of Fujian Province (FJNU),
Fujian Normal University, Fuzhou, 350007, P.R. China. \texttt{jianwang@fjnu.edu.cn}}

\date{}

\begin{abstract} We establish explicit quenched asymptotics for pure-jump symmetric L\'evy processes in general Poissonian potentials, which is closely related to large time asymptotic behavior of solutions to the nonlocal parabolic Anderson problem with Poissonian interaction. In particular, when the density function with respect to the Lebesgue measure of the associated L\'evy measure is given
by $$\rho(z)= \frac{1}{|z|^{d+\alpha}}\I_{\{|z|\le 1\}}+ e^{-c|z|^\theta}\I_{\{|z|> 1\}}$$ for some $\alpha\in (0,2)$, $\theta\in (0,\infty]$ and $c>0$, exact quenched asymptotics is derived for potentials with the shape function given by $\varphi(x)=1\wedge |x|^{-d-\beta}$ for $\beta\in (0,\infty]$ with $\beta\neq 2$. We also discuss quenched asymptotics in the critical case (e.g.,\, $\beta=2$ in the example mentioned above).

\medskip

\end{abstract}

\maketitle

\allowdisplaybreaks

\section{Background and main results}\label{section1}
This paper is devoted to the analysis of large time asymptotic behavior of solutions to the nonlocal parabolic Anderson problem with Poissonian interaction:
\begin{equation}\label{e:an1} \frac{\partial u}{\partial t}=Lu-V^\w  u\quad \end{equation} on $[0,\infty)\times \R^d$ with the initial condition $u(0,x)=1$. Here, $L$ is the infinitesimal generator of pure-jump symmetric L\'evy process $Z:=(Z_t,\Pp_x)_{t\ge0, x\in \R^d}$ with the characteristic exponent
\begin{equation}\label{e:sym}\psi(\xi)=\int_{\R^d\backslash\{0\}} (1-\cos \langle \xi,z\rangle )\,\nu(dz)\end{equation}for some symmetric L\'evy measure $\nu$ (i.e., $\nu$ is a Radon measure on $\R^d\backslash\{0\}$ that satisfies $\int_{\R^d\backslash\{0\}} (|z|^2\wedge 1)\,\nu(dz)<\infty$ and $\nu(A)=\nu(-A)$ for any $A\in \mathscr{B}(\R^d\backslash \{0\})$); the potential
\begin{equation}\label{e:pot}V^\w(x)=\int_{\R^d} \varphi(x-y)\,\mu^\w(dy),\end{equation} where $\mu^\w$ is a Poissonian random measure on $\R^d$ with density $\rho\,dx$, $\rho>0$, on a given probability space $(\Omega, \Q)$, and $\varphi$ is a non-negative shape function on $\R^d$. We refer to the monographs \cite{Ko,Sz2} for background on this topic. Throughout this paper, $\Q$ and $\Ee_\Q$ denote the probability  and the expectation, respectively, generated by the Poissonian field; while $\Pp_x$ and $\Ee_x$ denote the probability and the expectation, respectively, corresponding to the L\'evy process $Z$ with the starting point $x\in \R^d$.

Under mild assumptions (see Subsection \ref{section2.2}), the solution to the problem \eqref{e:an1} enjoys the Feynman-Kac representation
\begin{equation}\label{e:an2} u^\w(t,x)=\Ee_x\left[\exp\left(-\int_0^t V^\w(Z_s)\,ds\right)\right].\end{equation} Thus, the analysis of  properties for the solution to \eqref{e:an1} can be done via \eqref{e:an2} by estimating $u^\w(t,x)$. There are a number of works on the large time behavior of $u^\w(t,x)$ in both the annealed sense (averaged with respect to $\Q$) and the quenched sense (almost sure with respect to $\Q$). In this paper we will mainly analyse the quenched behavior of  $u^\w(t,x)$ for pure-jump symmetric L\'evy processes in Poissonian potentials with more general shape function $\varphi$.

Let us begin with recalling the history on related topics. The annealed asymptotics of $u^{\omega}(t,x)$ was first established by Donsker and Varadhan \cite{DV} for symmetric (but not necessarily isotropic) non-degenerate $\alpha$-stable processes (including Brownian motion). They proved in \cite[Theorem 3]{DV} that, when the shape function $\varphi(x)$ is of order $o(1/|x|^{d+\alpha})$ as $|x|\to\infty$, which is refereed to the light tailed case later,
\begin{equation}\label{e:a0}\lim_{t\to\infty} \frac{\log \Ee_\Q[u^{\omega}(t,x)]}{t^{d/(d+\alpha)}}=-\rho^{\alpha/(d+\alpha)}\left(\frac{d+\alpha}{\alpha}\right)\left(\frac{\alpha \lambda_{(\alpha)}(B(0,1))}{d}\right)^{d/(d+\alpha)},\end{equation} where
$$\lambda_{(\alpha)}(B(0,1))=\inf_{{\rm open}\,\, U, |U|=w_d}\lambda_1^{(\alpha)}(U),$$ $w_d$ is the volume of the unit ball $B(0,1)$, and $\lambda_1^{(\alpha)}(U)$ is the principle Dirichlet eigenvalue for the symmetric $\alpha$-stable process killed upon exiting $U$. In particular, when the symmetric $\alpha$-stable process is isotropic, it follows from the Faber-Krahn isoperimetric inequality that the infimum in the definition of $\lambda_{(\alpha)}(B(0,1))$ above is attained on the ball of radius $r_d=w_d^{-1/d}$ and so $ \lambda_{(\alpha)}(B(0,1))=w_d^{\alpha/d}\lambda_1^{(\alpha)}(B(0,1))$.  Then, in this case \eqref{e:a0} is reduced into
\begin{equation}\label{e:a0-1}\lim_{t\to\infty} \frac{\log \Ee_\Q[u^{\omega}(t,x)]}{t^{d/(d+\alpha)}}=-(\rho w_d)^{\alpha/(d+\alpha)}\left(\frac{d+\alpha}{\alpha}\right)\left(\frac{\alpha \lambda_1^{(\alpha)}(B(0,1))}{d}\right)^{d/(d+\alpha)}.\end{equation}  Later \^{O}kura \cite{Ok0} extended \cite[Theorem 3]{DV} to a large class of symmetric L\'evy processes whose exponent $\psi$ satisfies $\exp(-t\psi(\cdot)^{1/2})\in L^1(\R^d;dx)$ for all $t>0$ and can  be written as
\begin{equation}\label{e:a1} \psi(\xi)=\psi^{(\alpha)}(\xi)+o(|\xi|^\alpha),\quad |\xi|\to0\end{equation} for some $\alpha\in (0,2]$. Here, $\psi^{(\alpha)}(\xi)$ is the characteristic exponent of a symmetric non-degenerate $\alpha$-stable process $Z^{(\alpha)}$ (see \eqref{alpha} below) satisfying some kind of summability condition on $\psi^{(\alpha)}_*(\xi):=\inf_{t\ge1}t^\alpha\psi^{(\alpha)}(t^{-1}\xi)$; see Subsection \ref{section2.1} for more details.
 More explicitly, it was shown in \cite[Theorem 4.1]{Ok0} that \eqref{e:a0} still holds for symmetric L\'evy processes above with $\lambda_{(\alpha)}(B(0,1))$ defined via the principle Dirichlet eigenvalue for the killed symmetric $\alpha$-stable process $Z^{(\alpha)}$ with exponent $\psi^{(\alpha)}$ given in \eqref{e:a1}.

When the characteristic exponent of the L\'evy process $Z$ further satisfies
$$\psi(\xi)=O(|\xi|^\alpha),\quad |\xi|\to0,$$ and the shape function $\varphi$ fulfills $K:=\lim_{|x|\to\infty}{\varphi(x)}{|x|^{d+\beta}}\in (0,\infty)$ for some $0<\beta<\alpha$ (which is referred to the heavy tailed case), \^{O}kura  proved in \cite[Theorem 6.3']{Ok1} that
\begin{equation}\label{e:a2}\lim_{t\to\infty} \frac{\log \Ee_\Q[u^{\omega}(t,x)]}{t^{d/(d+\beta)}}=- \rho w_d \Gamma\left(\frac{\beta}{d+\beta}\right)K^{d/(d+\beta)}\end{equation} holds for symmetric L\'evy processes satisfying \eqref{e:a1}.  See Pastur \cite{Pa} for the first result on this direction when $Z$ is Brownian motion.
The reader also can be  referred to \cite[Theorem 6.4']{Ok1} and \cite[Theorem 1 and Remarks]{Ok2} for the study in the critical case, e.g., $K:=\lim_{|x|\to\infty}{\varphi(x)}{|x|^{d+\alpha}}\in (0,\infty)$; see the appendix for details.
In particular, according to all the conclusions above, the annealed asymptotics of $u^{\omega}(t,x)$ is of order $t^{-d/(d+\beta\wedge \alpha)}$ when $\varphi(x)=K(1\wedge |x|^{-d-\beta})$. However, in the light tailed case, the right hand side of \eqref{e:a0} for the annealed asymptotics of $u^{\omega}(t,x)$ is independent of $K$ and $\beta$, while in the heavy tailed case that of \eqref{e:a2} only depends on the constants $K$ and $\beta$.

Compared with the annealed asymptotics, the study of the quenched asymptotics of $u^{\omega}(t,x)$ is relatively limit. The first result for the quenched asymptotics of $u^{\omega}(t,x)$ for Brownian motions moving in a Poissonian potential was established by Sznitman in \cite[Theorem]{Sz}, which showed that when $\varphi$ is compactly supported (which in particular corresponds to the shape function $\varphi(x)=K(1\wedge |x|^{-d-\beta})$ with $\beta=\infty$, and so  belongs to the special light tailed case), $\Q$-almost surely for all $x\in \R^d$,
$$ \lim_{t\to\infty} \frac{\log   u^{\omega}(t,x) }{t/(\log t)^{2/d}}=-\left(\frac{\rho w_d}{d}\right)^{2/d} \lambda_{{\rm BM}}(B(0,1)),$$ where $\lambda_{{\rm BM}}(B(0,1))$ is the principle Dirichlet  eigenvalue for the Brownian motion killed upon exiting $B(0,1)$. More recently, the quenched asymptotics of $u^{\omega}(t,x)$ for symmetric L\'evy processes satisfying \eqref{e:a1} has been extensively studied in \cite{KP}; see \cite[Table 1 in p.\ 165]{KP} for results concerning explicit L\'evy processes.

Concerning Brownian motions in a heavy tailed Poissonian potential,  for example, $\varphi(x)=1\wedge |x|^{-d-\beta} $ with $\beta\in (0,2)$, it was shown in \cite[Theorem 2]{Fuk2} that $\Q$-almost surely
for any $x\in \R^d$, $$ \lim_{t\to\infty} \frac{\log   u^{\omega}(t,x) }{t/(\log t)^{\beta/d}}=-\frac{d}{d+\beta}\left(\frac{\beta}{d(d+\beta)}\right)^{\beta/d}\left(\rho w_d\Gamma\left(\frac{\beta}{d+\beta}\right)\right)^{(d+\beta)/d}.$$ (Indeed, the second order asymptotics was also proved in \cite{Fuk2}.)

However, the quenched asymptotics of $u^\w(t,x)$ for symmetric L\'evy processes in heavy tailed cases, as well as that in light tailed cases when $\varphi$ does not have compact support, are still unknown. The goal of this paper is to fill up these gaps. To state our main contribution, in the following two results we are restricted ourselves on the special but typical shape function $\varphi(x)= 1\wedge |x|^{-d-\beta}$ with $\beta\in (0,\infty]$.

\begin{theorem}\label{Th-1} Let $Z$ be a rotationally symmetric $\alpha$-stable process on $\R^d$ with $\alpha\in (0,2)$. Then,

\begin{itemize}
\item[(i)] When $\beta\in (\alpha,\infty]$, $\Q$-almost surely for all $x\in \R^d$,
\begin{align*}-(d+\alpha)^{\alpha/(d+\alpha)}\left[\left(\frac{\alpha}{d}\right)^{d/(d+\alpha)}+\left(\frac{d}{\alpha}\right)^{\alpha/(d+\alpha)}\right]A_1
 \le& \liminf_{t\to\infty} \frac{u^\w(t,x)}{t^{d/(d+\alpha)}}\\
 \le& \limsup_{t\to\infty} \frac{u^\w(t,x)}{t^{d/(d+\alpha)}}\\
  \le& -\alpha (\alpha+d/2)^{-d/(\alpha+d)}A_1,\end{align*} where
$$A_1= \left(\frac{\rho w_d}{d}\right)^{\alpha/(d+\alpha)} [\lambda_1^{(\alpha)}(B(0,1))]^{d/(d+\alpha)}.$$

\item[(ii)] When $\beta\in (0,\alpha)$, $\Q$-almost surely for all $x\in \R^d$,
\begin{align*}-(d+\alpha)^{\beta/(d+\beta)}\left[\left(\frac{\beta}{d}\right)^{d/(d+\beta)}+\left(\frac{d}{\beta}\right)^{\beta/(d+\beta)}\right]A_2
 \le & \liminf_{t\to\infty} \frac{u^\w(t,x)}{t^{d/(d+\beta)}}\\
 \le&\limsup_{t\to\infty} \frac{u(t,x)}{t^{d/(d+\beta)}}\le -\alpha^{\beta/(d+\beta)} A_2,\end{align*} where
$$A_2= \left(\frac{  d}{d+\beta}\right)^{\d/(\beta+d)}\left(\frac{\beta }{d(d+\beta)}\right)^{\beta/(\beta+d)} \Gamma\left(\frac{\beta}{d+\beta}\right)\rho w_d.$$ \end{itemize}
 \end{theorem}

\begin{theorem}\label{Thm1.2} Suppose that the L\'evy measure $\nu(dz)=\rho(|z|)\,dz$ satisfies
$$\rho(|z|)\asymp \frac{1}{|z|^{d+\alpha}}\I_{\{|z|\le 1\}}+ e^{-c|z|^\theta}\I_{\{|z|> 1\}}$$ for some $\alpha\in (0,2)$, $\theta\in (0,\infty]$ and $c>0$, where $f\asymp g$ means that there is a constant $c_0
\ge 1$ such that $c_0^{-1} g\le f\le c_0f.$  Then,
\begin{itemize}
\item[(i)] When $\beta\in (2,\infty]$, $\Q$-almost surely for all $x\in \R^d$,
$$ \lim_{t\to\infty} \frac{u^\w(t,x)}{t/(\log t)^{2/d}}= -\left(\frac{\rho w_d(1\wedge \theta)}{d}\right)^{2/d} \lambda_1^{(2)}(B(0,1)), $$ where $\lambda_1^{(2)}(B(0,1))$ is the first Dirichlet eigenvalue for the killed Brownian motion when exiting the ball $B(0,1)$ and with the covariance matrix $(a_{ij})_{1\le i,j\le d}$ as follows
$$a_{ij}=\int_{\R^d\backslash \{0\}}z_iz_j\,\nu(dz),\quad 1\le i,j\le d.$$

\item[(ii)] When $\beta\in (0,2)$, $\Q$-almost surely for all $x\in \R^d$,
$$ \lim_{t\to\infty} \frac{u^\w(t,x)}{t/(\log t)^{\beta/d}}=  -\frac{d}{d+\beta} \left(\frac{\beta(1\wedge \theta)}{d(d+\beta)}\right)^{\beta/d}\left[\rho w_d\Gamma\left(\frac{\beta}{d+\beta}\right)\right]^{(d+\beta)/d}.$$\end{itemize}
\end{theorem}

Theorems \ref{Th-1} and \ref{Thm1.2} show that the quenched asymptotics of $u^\w(t,x)$ for pure-jump symmetric L\'evy process in  Poissonian potentials depends not only on the shape function $\varphi$ in the potential, but also properties of the L\'evy measure (for large jumps) of the L\'evy process $Z$. This phenomenon happens for the annealed asymptotics of $u^\w(t,x)$, but there is much more involved in the quenched asymptotics. For instance, considering the example with $\theta\in (0,1)$ in Theorem \ref{Thm1.2} which satisfies \eqref{e:a1} with $\alpha=2$, in the light tailed case the precise value of the annealed asymptotics of $u^\w(t,x)$ is independent of $\theta$ by \eqref{e:a0-1}, but that of the quenched asymptotics of $u^\w(t,x)$ does depend on $\theta$ by Theorem \ref{Thm1.2}(i). The same occurs for the heavy tailed case. On the other hand, in both light tailed and heavy tailed cases, for rotationally  symmetric $\alpha$-stable processes, by Theorem \ref{Th-1} the correct order of the quenched asymptotics of $u^\w(t,x)$ is the same as that of the annealed asymptotics; however, according to Theorem \ref{Thm1.2}, it is not true for symmetric L\'evy processes with exponential decay for large jumps; see \cite[Section 1]{KP} for more discussions on this point in the light tailed setting.

Next, we briefly make  comments on our proofs for the quenched asymptotics of $u^\w(t,x)$ for pure-jump symmetric L\'evy process in general Poissonian potentials.
\begin{itemize}
\item[(i)] Compared with \cite{KP}, the crucial  ingredient to handle general light tailed cases is the observation that, due to the light tail of the potential, $ V^\w(x)$ is compared with $\tilde V^\w(x)$ whose associated shape function is compactly supported. This enables us to use the classical approach in \cite{Sz,KP}; that is, when the shape function has compact support, $\Q$-almost surely there exists a large area where the potential is zero and so the principle Dirichlet eigenvalue of the process $Z$ itself is naturally involved in the quenched asymptotics of $u^\w(t,x)$.  When $\beta=\infty$ (this is just the case that the shape function $\varphi$ has compact support), Theorems \ref{Th-1}(i) and \ref{Thm1.2}(i) have been proven in \cite{KP}; see \cite[Table 1 in p.\ 165]{KP} for more details. Based on this and the strategy of the approach mentioned above, we believe that assertions of \cite{KP} should hold true for all light tailed cases.
\item[(ii)] In heavy tailed cases, the potential $V^\w(x)$ will play a dominated role in the  quenched asymptotics of $u^\w(t,x)$. Similar to the Brownian motion case studied in \cite{Fuk2}, it is natural to expect that the main contribution of $u^\w(t,x)$ defined by \eqref{e:an2} comes from the process $Z$ which spends most of the time in the area where $V^\w(x)$ takes small value. Motivated by the fact, we partly adopt the argument in \cite{Fuk2} to treat upper bounds of the principle Dirichlet eigenvalue for the random Schr\"{o}dinger operator associated with the equation \eqref{e:an1}, which in turn yield explicit  quenched asymptotics of $u^\w(t,x)$ in general heavy tailed setting.
    \item[(iii)] To consider quenched asymptotics for pure-jump symmetric L\'evy processes in both light tailed and heavy tailed potentials at the same time, we give an unified approach which is inspired by \cite{CX} (which studied quenched asymptotics for Brownian motions in renormalized Poissonian potentials) and based on recent development on (Dirichlet) heat kernel estimates for symmetric jump processes. We emphasize that the argument of lower bounds for quenched asymptotics of $u^\w(t,x)$ here is different from that in \cite{KP}. In particular,  the lower bound for the quenched asymptotics of $u^\w(t,x)$ in Theorem \ref{Th-1}(i) for symmetric rotationally $\alpha$-stable process slightly improves that in \cite{KP}; see \cite[Remark 5.1(4)]{KP}.
\end{itemize}

We further mention that our main results for quenched estimates of $u^\w(t,x)$ hold (see Theorems \ref{T:upper} and \ref{T-low}) for pure-jump symmetric L\'evy process in general Poissonian potentials, so the results should apply various examples discussed in \cite[Section 5]{KP}. It is also possible to extend them to symmetric L\'evy processes with non-degenerate Brownian motion as done in \cite{KP}, and the details are left to interested readers.
Instead, to highlight the power of our approaches, we will present the quenched estimates of $u^\w(t,x)$ with critical potentials (for example, $\varphi(x)= 1\wedge |x|^{-d-\alpha}$ with $\alpha$ being in \eqref{e:a1}) in the appendix. Specially, we can prove that
\begin{proposition}\label{P:cre}
\begin{itemize}
\item[(i)] Let $Z$ be a rotationally  symmetric $\alpha$-stable process on $\R^d$ with $\alpha\in (0,2)$, and $\varphi(x)=1\wedge |x|^{-d-\alpha}$.  Then,
$\Q$-almost surely for all $x\in \R^d$,
$$-\infty  <\liminf_{t\to\infty} \frac{u^\w(t,x)}{t^{d/(d+\alpha)}}
 \le  \limsup_{t\to\infty} \frac{u^\w(t,x)}{t^{d/(d+\alpha)}}<0.$$

 \item[(ii)] Let $Z$ be a pure-jump rotationally symmetric L\'evy process given in Theorem $\ref{Thm1.2}$, and $\varphi(x)=1\wedge |x|^{-d-2}$.  Then,
$\Q$-almost surely for all $x\in \R^d$,
$$-\infty < \liminf_{t\to\infty}\frac{u^\w(t,x)}{t/(\log t)^{2/d}}
 \le  \limsup_{t\to\infty} \frac{u^\w(t,x)}{t/(\log t)^{2/d}}<0.$$\end{itemize}\end{proposition}

\ \

The rest of the paper is arranged as follows. In Section \ref{section2}, we give some preliminaries and main assumptions of our paper. Section \ref{section3} is the main part of our paper, and it is split into three subsections. In particular, after establishing quenched bounds for $u^\w(t,x)$ and estimates for the principle Dirichlet eigenvalue, we will derive general quenched estimates of $u^\w(t,x)$ in here. Section \ref{section4} is devoted to proofs of Theorems \ref{Th-1} and \ref{Thm1.2}. Finally, in the appendix we present  upper quenched bounds for the principle Dirichlet eigenvalue in the heavy tailed case, and  quenched estimates of $u^\w(t,x)$ with critical potentials.

\section{Preliminaries and Assumptions}\label{section2}
\subsection{L\'evy processes}\label{section2.1}
Let $Z:=(Z_t,\Pp_x)_{t\ge0, x\in \R^d}$ be a pure-jump symmetric L\'evy process on $\R^d$ with the characteristic exponent $\psi$ given by \eqref{e:sym}. Throughout the paper, we will assume the following two conditions hold for the exponent $\psi$:
\begin{itemize}
\item[(i)] $e^{-t\psi^{1/2}(\cdot)}\in L^1(\R^d;dx)$ for all $t>0$;
\item[(ii)] \begin{equation}\label{e:a1-1}\psi(\xi)=\psi^{(\alpha)}(\xi)+o(|\xi|^\alpha),\quad |\xi|\to0\end{equation} for some $\alpha\in (0,2]$, where
\begin{equation}\label{alpha}\psi^{(\alpha)}(\xi)=\begin{cases} \int_0^\infty \int_{\Ss^{d-1}} \frac{1-\cos(r\langle \xi,z\rangle)}{r^{1+\alpha}}\,\mu(dz)\,dr,&\quad \alpha\in (0,2),\\
\sum_{1\le i,j\le d} a_{ij}\xi_i\xi_j,&\quad \alpha=2\end{cases}\end{equation} with $\mu$ being a symmetric finite measure on the unit sphere $\Ss^{d-1}$ and $(a_{ij})_{1\le i,j\le d}$ being a symmetric non-negative definite matrix. Moreover, $\inf_{|\xi|=1}\psi^{(\alpha)}(\xi)>0, $ and, for each $\delta, r>0$,
$$\sum_{\xi\in r\Z^d}\exp(-\delta \psi^{(\alpha)}_*(\xi))<\infty,$$ where  $\psi^{(\alpha)}_*(\xi)=\inf_{t\ge1}t^\alpha\psi^{(\alpha)}(t^{-1}\xi)$.
\end{itemize}
It is clear that under (i) the process $Z$ has the transition density function $p(t,x-y)=p(t,x,y)$  with respect to the Lebesgue measure such that $p(t,0)=\sup_{x\in \R^d} p(t,x)<\infty$ for all $t>0$. We further suppose that $p(t,x)$ is strictly positive for all $t>0$ and $x\in \R^d$.
Note that, the asymptotic condition \eqref{e:a1-1} of $\psi(\xi)$ is essentially based on the property of the L\'evy measure $\nu$ on $\{z\in \R^d: |z|>1\}$. For example, according to \cite[Proposition 5.2(i)]{KP}, if $\nu$ has finite second moment, i.e., $\int_{\{|z|>1\}}|z|^2\,\nu(dz)<\infty$, then \eqref{e:a1-1} holds with $\alpha=2$ and
$a_{ij}=\frac{1}{2}\int_{\R^d\backslash\{0\}} z_iz_j\,\nu(dz).$

In this paper, we always let $D$ be a bounded  domain (i.e.,\ connected open set) of $\R^d$. Let $Z^D:=(Z_t^D, \Pp^x)_{t\ge 0,x\in D}$ be the subprocess of $Z$ killed upon exiting $D$. Then, $Z^D$ has the transition density function
$$p^D(t,x,y)=p(t,x,y)-\Ee_x \left(p(t-\tau_D,Z_{\tau_D},y)\I_{\{\tau_D\le t\}}\right),\quad t>0, x,y\in D,$$ where $\tau_D=\inf\{t>0:Z_t\notin D\}$. Denote by $(P_t^D)_{t\ge0}$ the Dirichlet semigroup associated with the process $Z^D$. Since $D$ is bounded and $p^D(t,x,y)\le p(t,x,y)=p(t,x-y)\le p(t,0)<\infty$ for all $t>0$ and $x,y\in D$, the operators $P_t^D$ are compact and admit a sequence of positive eigenvalues
$$0<\lambda_1(U)<\lambda_2(U)\le \lambda_3(U)\le \cdots \to \infty.$$ When $Z$ is a symmetric $\alpha$-stable process with $\alpha\in (0,2]$, the eigenvalues will be denoted by
$ \lambda_i^{(\alpha)}(U)$ for $i\ge1$.

\subsection{Random potential} \label{section2.2}Consider the random potential $V^\w$ given by \eqref{e:pot}, which can be written as
$$V^\w(x)=\sum_{i}\varphi(x-\w_i),\quad x\in \R^d,$$ and the points $\{\w_i\}$ are from a realization of a homogeneous Poisson point process in $\R^d$ with parameter $\rho>0$. In this paper, we assume that the non-negative shape function
$\varphi$ is continuous, and satisfies
\begin{equation}\label{e:func}\int_{\R^d} \left(e^{\bar\varphi(x)}-1\right)\,dx<\infty,\end{equation} where $\bar\varphi(x)=\sup_{z\in B(x,1)}\varphi(z)$. Then, following the proof of \cite[Lemma 5]{Fuk2}, we know that
$\Q$-almost surely there is $r(\w)>0$ such that for all $r\ge r(\w)$,
\begin{equation}\label{e:upper-v}\sup_{x\in B(0,r)}V^\w(x)\le 3d \log r.\end{equation} A typical example that satisfies \eqref{e:func} is the function $\varphi(x)=K(1\wedge |x|^{-d-\theta})$ for some positive constants $K$ and $\theta$. Indeed, if there are constants $c_0,\theta>0$ such that
$$\varphi(x)\le\frac{ c_0}{(1+|x|)^{d+\theta}},\quad x\in \R^d,$$ then, according to \cite[Lemma 2.1]{CM}, we even have that $\Q$-almost surely there is  $r(\w)>0$ so that for all $r\ge r(\w)$,
$$\sup_{x\in B(0,r)}V^\w(x)\le c\left(1+\frac{\log r}{\log \log r}\right),$$ where $c>0$ is independent of $r(\w)$ and $r$. In particular, \eqref{e:upper-v} yields that for $\Q$-almost surely,
$V^\w$ belongs to the local Kato class relative to the process $Z$, i.e., $\Q$-almost surely,
$$\lim_{t\to 0}\sup_{x\in \R^d}\int_0^t \Ee_x(V^\w(Z_s)\I_{\{Z_s\in B(0,R)\}})\,ds=0$$ for all $R>0.$

\subsection{Feynman-Kac semigroup}  Since, $\Q$-almost surely,
$V^\w$ belongs to the local Kato class relative to the process $Z$, we can well define the random Feynman-Kac semigroups $(T_t^{V^\w})_{t\ge0}$ and $(T_t^{V^\w,D})_{t\ge0}$ as follows
\begin{align*}T_t^{V^\w}f(x)=&\Ee_x\left[f(Z_t)e^{-\int_0^t V^\w(Z_s)\,ds}\right],\quad f\in L^2(\R^d;dx),t>0,\\
T_t^{V^\w, D}f(x)=&\Ee_x\left[f(Z_t)e^{-\int_0^t V^\w(Z_s)\,ds}\I_{\{\tau_D>t\}}\right],\quad f\in L^2(D;dx),t>0.\end{align*} Under our setting, both $(T_t^{V^\w})_{t\ge0}$ and $(T_t^{V^\w,D})_{t\ge0}$ admit strictly positive and bounded symmetric kernels $p^{V^{\w}}(t,x,y)$ and  $p^{V^{\w}, D}(t,x,y)$ with respect to the Lebesgue measure  respectively, such that
$$ p^{V^{\w}}(t,x,y)\le p(t,x,y)=p(t,x-y),\quad x,y\in \R^d,t>0,$$ and
$$p^{V^{\w}, D}(t,x,y)\le p^D(t,x,y)\le p(t,x-y),\quad x,y\in D, t>0.$$

On the other hand, it is known that $(T_t^{V^\w})_{t\ge0}$ can be generated by the random nonlocal Schr\"{o}dinger operator $-H^\w$ with $H^{\w}:=-L+V^\w$, where $L$ is the infinitesimal generator of the L\'evy process $Z$. Hence, the semigroup $(T_t^{V^\w,D})_{t\ge0}$ corresponds to  the Schr\"{o}dinger operator $-H^\w$ with the Dirichlet conditions on $D^c$. In particular, the operators $T_t^{V^\w,D}$  are compact, so that $\Q$-almost surely the spectrum of the operator $-H^\w$  with the Dirichlet conditions on $D^c$ is discrete:
$$  0<\lambda_1^{V^\w,D}<\lambda_2^{V^\w,D}\le \lambda_3^{V^\w,D}\le \cdots \to \infty.$$ For simplicity, below we write $\lambda_1^{V^\w,D}$ as $\lambda_{V^\w,D}$, which will play an important role in our paper.
It further follows that
\begin{equation}\label{e:note1ss}\|T_t^{V^\w,D}\|_{L^2(D;dx)\to L^2(D;dx)}\le e^{-t\lambda_1^{V^\w,D} }=e^{-t\lambda_{V^\w,D}},\quad t>0\end{equation} and that for any  $x\in D$,
\begin{equation}\label{e:note1-}\Ee_x\left[\exp\left(-\int_0^t V^{\w}(Z_s)\,ds\right) \delta_x(Z_t):\tau_D> t\right]= \sum_{k=1}^\infty e^{-t\lambda_k^{V^\w,D}}e_k(x)^2,\quad t>0,\end{equation}  where  $\|\cdot\|_{L^2(D;dx)\to L^2(D;dx)}$ is denoted by the operator norm from $L^2(D;dx)$ to $L^2(D;dx)$, and  $\{e_k(x)\}_{k\ge1}$ are the eigenfunctions corresponding to $\{\lambda_k^{V^\w,D}\}_{k\ge1}$ respectively with $\|e_k\|_{L^2(D;dx)}=1$ for all $k\ge1$.  According to \eqref{e:note1ss}, it then holds that
\begin{equation}\label{e:note3} \int_D\Ee_x\left[\exp\left(-\int_0^t V^{\w}(Z_s)\,ds\right):\tau_D> t\right]\,dx\le |D|  e^{-t\lambda_{V^\w,D}},\quad t>0.\end{equation} Thanks to \eqref{e:note1-}, we also have
\begin{equation}\label{e:note2}e^{-t\lambda_{V^\w,D}}\le \int_D\Ee_x\left[\exp\left(-\int_0^t V^{\w}(Z_s)\,ds\right) \delta_x(Z_t):\tau_D> t\right]\,dx,\quad t>0.\end{equation}

\section{General bounds for quenched asymptotics of $u^{\omega}(t,x)$}\label{section3}
In this section, we establish general bounds for quenched asymptotics of $u^{\omega}(t,x)$. Let $Z$ be a pure-jump symmetric L\'evy process on $\R^d$ and $V^\w$ be the random potential given by \eqref{e:pot}, both of which satisfy all the assumptions in the previous section.
For the index $\alpha\in (0,2]$ given in \eqref{e:a1-1}, we will consider the following two cases.

\begin{itemize}
\item {\it light tailed case}\,\,{\rm (L)}:\,\, The shape function $\varphi$ in the random potential $V^\w(x)$ satisfies that
$$\lim_{|x|\to\infty} \varphi(x) |x|^{d+\alpha}=0.$$

\item {\it heavy tailed case}\,\,{\rm (H)}:\,\, The characteristic exponent $\psi(\xi)$ of the process $Z$ fulfills that $\psi(\xi)=O(|\xi|^\alpha)$ as $|\xi|\to0$, and there are constants $\beta\in (0,\alpha)$ and $K>0$ such that, for the shape function $\varphi$ in the random potential $V^\w(x)$, it holds that
\begin{equation}\label{:h}\lim_{|x|\to\infty} \varphi(x) |x|^{d+\beta}=K.\end{equation}

\end{itemize}

The section is split into three parts. We first show quenched bounds for $u^{\omega}(t,x)$, and then present estimates for the principle Dirichlet eigenvalue $\lambda_{V^\w,D}$. General explicit results for quenched estimates of $u^{\omega}(t,x)$ are given in Subsection \ref{sub-3.3}.

Because of the homogeneities of the L\'evy process $Z$ and the potential $V^\w$, the distribution of the quenched bounds for $u^{\omega}(t,x)$ in this section does not depend on the starting point, and so, without loss of generality, we can take $x=0$ in the proof.

\subsection{Quenched bounds for $u^{\omega}(t,0)$}
In this part, we derive some pointwise quenched bounds for $u^\w(t,0)$. Some of arguments below are motivated by those in \cite[Section 4]{CX}.
\subsubsection{\bf Upper bounds}

\begin{proposition}\label{L:lim-u} For any bounded domain $D$,  $0<\delta<t$, $R>0$ and $a>1$, and for any $\w\in \Omega$,
\begin{align*} u^\w(t,0)
 \le  \Pp_0(\tau_D \le t)
  +\min\bigg\{& p(\delta,0)^{1/2}  |D|^{1/2}\exp(-(t-\delta/2)\lambda_{V^\w,D}), \\
&
 p(\delta,0)^{1/a} |D|^{1/a}\exp(-a^{-1}(t-\delta)\lambda_{aV^\w,D}) \bigg\}.\end{align*} \end{proposition}
\begin{proof}  We mainly follow the idea of \cite[Lemma 2.1]{Fu}. For any bounded domain $D$, $t>0$, and $\w\in \Omega$,
\begin{align*}u^\w(t,0)&=\Ee_0\left[\exp\left(-\int_0^t V^\w(Z_s)\,ds\right)\right]\\
&\le \Ee_0\left[\exp\left(-\int_0^t V^\w(Z_s)\,ds\right): \tau_{D}> t\right]+\Pp_0(\tau_{D}\le t)\\
&=:I_1+I_2.\end{align*}
Next, we will estimate $I_1$ in two different ways.

First, we repeat the proof of \cite[Lemma 3.1]{KP} as follows. For any $0<\delta<t$,
\begin{align*}I_1&=T_t^{V^\w,D}\I_{D}(0)=T_{\delta/2}^{V^\w,D}T_{t-\delta/2}^{V^\w,D}\I_{D}(0)\\
&=\langle p^{V^\w,D}(\delta/2,0,\cdot), T_{t-\delta/2}^{V^\w,D}\I_{D}\rangle_{L^2(D;dx)}\\
&\le \|p^{V^\w,D}(\delta/2,0,\cdot)\|_{L^2(D;dx)} \|T_{t-\delta/2}^{V^\w,D}\I_{D}\|_{L^2(D;dx)}\\
&\le \|p (\delta/2,0,\cdot)\|_{L^2(\R^d;dx)}e^{-(t-\delta/2)\lambda_{V^\w,D}}\|\I_{D}\|_{L^2(D;dx)}\\
&=p (\delta,0)^{1/2} |D|^{1/2}\exp(-(t-\delta/2)\lambda_{V^\w,D }),\end{align*} where in the first inequality we used the Cauchy-Schwarz inequality and the second inequality follows from \eqref{e:note1ss}.

Second, for any $0<\delta<t$, by the H\"{o}lder inequality with $a,b>1$ satisfying $1/a+1/b=1$,
\begin{align*}I_1&\le \left(\Ee_0\left[\exp\left(-b\int_0^\delta V^\w(Z_s)\,ds\right)\right]\right)^{1/b}\left(\Ee_0\left[\exp\left(-a\int_\delta ^tV^\w(Z_s)\,ds\right):\tau_{D}> t\right]\right)^{1/a}\\
&\le \left(\int_{D} p^{D}(\delta,0,x)\Ee_x\left[\exp\left(-a\int_0 ^{t-\delta}V^\w(Z_s)\,ds\right):\tau_{D}> t-\delta\right]\,dx\right)^{1/a}\\
&\le p(\delta,0)^{1/a}\left(\int_{D}\Ee_x\left[\exp\left(-a\int_0 ^{t-\delta}V^\w(Z_s)\,ds\right):\tau_{D}> t-\delta\right]\,dx\right)^{1/a}\\
&\le p(\delta,0)^{1/a} |D|^{1/a}\exp(-a^{-1}(t-\delta)\lambda_{aV^\w, D}), \end{align*} where in the last inequality we used \eqref{e:note3}.

Therefore, the assertion follows from all the estimates above.
\end{proof}

\begin{remark}The proof above essentially claims that for any bounded domain $D$, $0<\delta<t$, $a>1$ and $\w\in \Omega$,
$$\Ee_0\left[\exp\left(-\int_0^t V^\w(Z_s)\,ds\right): \tau_{D}> t\right]\le I(D,t, V^\w,\delta,a),$$
where
\begin{align*}I(D,t, V^\w,\delta,a):=\min\bigg\{& p(\delta,0)^{1/2}  |D|^{1/2}\exp(-(t-\delta/2)\lambda_{V^\w,D}), \\
&
 p(\delta,0)^{1/a} |D|^{1/a}\exp(-a^{-1}(t-\delta)\lambda_{aV^\w,D}) \bigg\}.\end{align*} By this estimate, we can further make a slightly improving argument for  Proposition \ref{L:lim-u}. Indeed, let $\{D_k\}_{k \ge1}$ be a sequence of increasing bounded domains such that $\cup_{k\ge1}D_k=\R^d$. Then, for any $t>0$,
\begin{align*}u^\w(t,0)&\le \Ee_0\left[\exp\left(-\int_0^t V^\w(Z_s)\,ds\right): \tau_{D_1}> t\right]\\
&\quad + \sum_{k=1}^\infty\Ee_0\left[\exp\left(-\int_0^t V^\w(Z_s)\,ds\right): \tau_{D_{k}}\le t< \tau_{D_{k+1}}\right]\\
&=:J_0+\sum_{k=1}^\infty J_k.\end{align*}
It is clear that
$J_0\le I(D_1,t, V^\w,\delta,a).$ On the other hand, by the H\"{o}lder inequality, for any $k\ge1$ and $\xi,\eta>1$ with $1/\xi+1/\eta=1$,
\begin{align*}J_k\le & \left[\Pp_0(\tau_{D_{k}}\le t< \tau_{D_{k+1}})\right]^{1/\xi} \left[\Ee_0\left(\exp\left(-\eta\int_0^t V^\w(Z_s)\,ds\right):\tau_{D_{k+1}}>t\right)\right]^{1/\eta}\\
\le& \left[\Pp_0(\tau_{D_{k}}\le t< \tau_{D_{k+1}})\right]^{1/\xi}  \left[I(D_{k+1},t, \eta V^\w,\delta,a)\right]^{1/\eta}.
\end{align*} Therefore, it holds that
$$ u^\w(t,0)\le I(D_1,t, V^\w,\delta,a)+\sum_{k=1}^\infty \left[\Pp_0(\tau_{D_{k}}\le t< \tau_{D_{k+1}})\right]^{1/\xi}  \left[I(D_{k+1},t, \eta V^\w,\delta,a)\right]^{1/\eta}.$$
In particular, letting $\eta\to \infty$ (i.e., $\xi\to1$), the estimate above is reduced to Proposition \ref{L:lim-u}.
 \end{remark}

\subsubsection{\bf Lower bounds}

\begin{lemma}\label{L:lem-l} For any bounded domain $D\subset \R^d$, $0<\delta<t$ and $a,b>1$ with $1/a+1/b=1$, and for any $\w\in \Omega$,
\begin{align*}&\int_D \Ee_x\left[\exp\left(-\int_0^{t} V^\w(Z_s)\,ds\right):\tau_D> t\right]\,dx\\
&\ge p(\delta,0)^{-1}  p(t,0)^{-ab^{-1}} |D|^{-2ab^{-1}}  \,\exp\left(-a (t+\delta)\lambda_{a^{-1} V^\w, D}\right). \end{align*}
 \end{lemma}

\begin{proof} We start from \eqref{e:note2}, i.e.,
$$e^{-t\lambda_{V^\w, D}} \le \int_D \Ee_x\left[\exp\left(-\int_0^t V^\w(Z_s)\,ds\right) \delta_x(Z_t):\tau_D> t\right]\,dx,\quad t>0.$$
Replacing $t$ and $V^\w$ by $t+\delta$ and $a^{-1}V^\w$ respectively in the inequality above, we get by the H\"{o}lder inequality that for all $a,b>1$ with $1/a+1/b=1$,
\begin{align*}e^{-(t+\delta) \lambda_{a^{-1}V^\w, D}}&\le \int_D \Ee_x\left[\exp\left(-a^{-1}\int_0^{t+\delta} V^\w(Z_s)\,ds\right) \delta_x(Z_{t+\delta}):\tau_D> t+\delta\right]\,dx\\
&\le \left(\int_D \Ee_x\left[\exp\left(-\int_0^{t} V^\w(Z_s)\,ds\right) \delta_x(Z_{t+\delta}):\tau_D> t+\delta\right]\,dx\right)^{1/a}\\
&\quad\times \left(\int_D \Ee_x\left[\exp\left(-\frac{b}{a}\int_t^{t+\delta} V^\w(Z_s)\,ds\right):\tau_D> t+\delta\right]\,dx\right)^{1/b}\\
&=:I_1\times I_2.\end{align*}

On the one hand, by the Markov property,
\begin{align*}I_1&\le \left(\int_D \Ee_x\left[\exp\left(-\int_0^{t} V^\w(Z_s)\,ds\right) \delta_x(Z_{t+\delta}):\tau_D> t\right]\,dx\right)^{1/a}\\
&\le \left(\int_D \Ee_x\left[\exp\left(-\int_0^{t} V^\w(Z_s)\,ds\right)p(\delta,x-Z_t):\tau_D> t \right]\,dx\right)^{1/a}\\
&\le p(\delta,0)^{1/a} \left(\int_D \Ee_x\left[\exp\left(-\int_0^{t} V^\w(Z_s)\,ds\right):\tau_D> t \right]\,dx\right)^{1/a}.  \end{align*}
On the other hand, also due to the Markov property,
\begin{align*}I_2= & \left(\int_D \int_D p^D(t,x,y) \Ee_y\left[\exp\left(-\frac{b}{a}\int_0^{\delta} V^\w(Z_s)\,ds\right):\tau_D> \delta\right]\,dy\,dx\right)^{1/b}\\
\le& p(t,0)^{1/b}|D|^{1/b} \left( \int_D \Ee_y\left[\exp\left(-\frac{b}{a}\int_0^{\delta} V^\w(Z_s)\,ds\right):\tau_D> \delta\right]\,dy\right)^{1/b}\\
\le& p(t,0)^{1/b} |D|^{2/b}.\end{align*}

Combining with both estimates above, we find that
\begin{align*}&\int_D \Ee_x\left[\exp\left(-\int_0^{t} V^\w(Z_s)\,ds\right):\tau_D> t\right]\,dt\\
&\ge p(\delta,0)^{-1} p(t,0)^{-a/b} |D|^{-2a/b}   \,\exp\left( -a (t+\delta)\lambda_{a^{-1} V^\w, D}\right). \end{align*} The proof is completed.
\end{proof}

\begin{proposition}\label{P:prol} For any bounded domain $D\subset \R^d$ with $0\in D$, subdomain $D_1\subset D$, $0<\delta<t$, $a,b>1$ with $1/a+1/b=1$ and for any $\w\in \Omega$,
\begin{align*}u^\w(t,0)&\ge\Ee_0\left[\exp\left(-\int_0^t V^\w(Z_s)\,ds\right):\tau_D> t\right]\\
&\ge p(\delta,0)^{-a}p({t-\delta},0)^{-a^2/b}  |D_1|^{-2a^2/b}  \left(\Ee_0\left[\exp\left(\frac{b}{a} \int_0^\delta V^\w(Z_s)\,ds\right):\tau_D> \delta\right]\right)^{-a/b}
\\
&\quad\times\left(\inf_{x\in D_1} p^D(\delta,0,x)\right)^a \exp\left(-a^2 t\lambda_{a^{-2} V^\w, D_1}\right). \end{align*}

 \end{proposition}

\begin{proof} For $0<\delta<t$, by the H\"{o}lder inequality, we have that for any $a,b>1$ with $1/a+1/b=1$,
\begin{align*}&\Ee_0\left[\exp\left(-a^{-1}\int_\delta^t V^\w(Z_s)\,ds\right):\tau_D> t\right]\\
&\le \left(\Ee_0\left[\exp\left(-\int_0^t V^\w(Z_s)\,ds\right):\tau_D> t\right]\right)^{1/a}\\
&\quad\times \left(\Ee_0\left[\exp\left( \frac{b}{a}\int_0^\delta V^\w(Z_s)\,ds\right):\tau_D> \delta\right]\right)^{1/b}.\end{align*}

Note that, according to the Markov property,
\begin{align*}&\Ee_0\left[\exp\left(-a^{-1}\int_\delta^t V^\w(Z_s)\,ds\right):\tau_D> t\right]\\
&=\int_D p^D(\delta,0,x)\Ee_x\left[\exp\left(-a^{-1}\int_0^{t-\delta} V^\w(Z_s)\,ds\right):\tau_D> t-\delta\right]\,dx\\
&\ge \left(\inf_{x\in D_1} p ^D(\delta, 0,x)\right)\int_{D_1}\Ee_x\left[\exp\left(-a^{-1}\int_0^{t-\delta} V^\w(Z_s)\,ds\right):\tau_{D_1}> t-\delta\right]dx.\end{align*}

Hence,
\begin{align*}&\Ee_0\left[\exp\left(-\int_0^t V^\w(Z_s)\,ds\right):\tau_D> t\right]\\
&\ge\left(\inf_{x\in D_1} p^D(\delta,0,x)\right)^a\left(\int_{D_1}\Ee_x\left[\exp\left(-a^{-1}\int_0^{t-\delta} V^\w(Z_s)\,ds\right):\tau_{D_1}> t-\delta\right]\,dx \right)^a\\
&\quad\times\left(\Ee_0\left[\exp\left( \frac{b}{a}\int_0^\delta V^\w(Z_s)\,ds\right):\tau_D> \delta\right]\right)^{-a/b}\\
&\ge \left(\inf_{x\in D_1} p ^D(\delta, 0,x)\right)^a p(\delta,0)^{-a}p (t-\delta, 0)^{-a^2b^{-1}} |D_1|^{-2a^2b^{-1}}
\\
&\quad\times \exp\left(-a^2 t\lambda_{a^{-2} V^\w, D_1}\right)\,\left(\Ee_0\left[\exp\left( \frac{b}{a}\int_0^\delta V^\w(Z_s)\,ds\right):\tau_D> \delta\right]\right)^{-a/b}, \end{align*} where in the last inequality we used Lemma \ref{L:lem-l}. The proof is finished. \end{proof}

\subsection{Estimates for the principle Dirichlet eigenvalue}

 In order to apply Propositions \ref{L:lim-u} and \ref{P:prol} to obtain explicit quenched asymptotics for $u^\w(t,0)$, we need to estimate the principle Dirichlet eigenvalue $\lambda_{ V^\w, D}$.

It was known that the  large time asymptotic behavior of solutions to \eqref{e:an1} is closely connected to the integrated density of states of the random Schr\"{o}dinger operator $H^\w=-L+V^\w$, which is defined by
\begin{equation}\label{e:density}N(\lambda)=\lim_{R\to \infty} \frac{1}{(2R)^d}\Ee_\Q\left[\sharp\{k\in \N: \lambda_k^{V^\w,B(0,R)}\le \lambda\}\right] \end{equation} with $\lambda_k^{V^\w,B(0,R)}$ being the $k$-th smallest eigenvalue of $H^\w$  with the Dirichlet conditions on $B(0,R)^c$.  See \cite[Section 5]{Ok1} for the existence of the limit above. Indeed, the existence of the limit in \eqref{e:density} was proved by using the spatial superadditivity property of $\Ee_\Q\left[\sharp\{k\in \N: \lambda_k^{V^\w,B(0,R)}\le \lambda\}\right]$, and so it is in fact the supremum over $R>0$.
Furthermore, it was observed that $N(\lambda)$ is the Laplace transform of the expectation of $u^\w(t,x)$ given by \eqref{e:an2} on $(\Omega,\Q)$. Then, an appropriate Tauberian theorem can be used to derive the information on the tail of $N(\lambda)$ as $\lambda\to0$ from the large time behavior of $\Ee_\Q[u^{\omega}(t,x)]$. Due to the corresponding Abelian theorem, the converse is also true.
We note that the study of $N(\lambda)$ requires the use of the associated pinned process rather than the symmetric L\'evy process $Z$ itself.

\subsubsection{\bf Lower bounds of $\lambda_{V^\w,B(0,R)}$ for $R$ large enough}\label{section3.1}

To estimate lower bounds of $\lambda_{V^\w,B(0,R)}$, we now recall some known results about the integral density $N(\lambda)$  of states of the random Schr\"{o}dinger operator $H^\w=-L+V^\w$  defined by \eqref{e:density}. It has been proved in \cite[Theorems 6.2 and 6.3]{Ok1} that
$$\lim_{\lambda \to0} \lambda ^{d/(\beta\wedge \alpha)} \log N (\lambda) =-k_0,$$ where
\begin{equation}\label{e:con-}
k_0:=\begin{cases} \rho \lambda_{(\alpha)}(B(0,1))^{d/\alpha},& \hbox{case (L)},\\
\frac{\beta}{d+\beta}\left(\frac{d}{d+\beta}\right)^{d/\beta}\!\!\left(\Gamma\left(\frac{\beta}{d+\beta}\right)\rho w_d\right)^{(d+\beta)/\beta}\!\!K^{d/\beta},&\hbox{case (H)},\end{cases} \end{equation} where $$\lambda_{(\alpha)}(B(0,1))=\inf_{{\rm open}\,\, U, |U|=w_d}\lambda_1^{(\alpha)}(U),$$$w_d$ is the volume of the unit ball $B(0,1)$, and $\lambda_1^{(\alpha)}(U)$ is the principle Dirichlet eigenvalue for the symmetric $\alpha$-stable process killed upon exiting the open subset $U$ and with the exponent $\psi^{(\alpha)}(\xi)$ given in \eqref{e:a1-1}. In particular, when this symmetric $\alpha$-stable process is isotropic, $\lambda_{(\alpha)}(B(0,1))=w_d^{\alpha/d}\lambda_1^{(\alpha)}(B(0,1))$.
With this at hand, we can see from the arguments of (2.3)--(2.6) in \cite[Section 2]{Fu} that for any $\varepsilon\in (0,1)$, $\Q$-almost surely  there is $R_\varepsilon(\w)>0$ such that for every $R\ge R_\varepsilon(\w)$,
\begin{equation}\label{e:lim-u}\lambda_{V^\w,B(0,R)}\ge (1-\varepsilon) \left(\frac{k_0}{d \log R}\right)^{(\alpha\wedge \beta)/d}.\end{equation}

\subsubsection{\bf Upper bounds of $\lambda_{V^\w,B(z,r)}$ for $r$ large enough with some $z$}

The following proposition is crucial for lower bounds  of quenched asymptotic of $u^\w(t,0)$.

\begin{proposition}\label{L:lem-l3} The following two statements hold.
\begin{itemize}
\item[(i)] In the light tailed case {\rm(L)},  for any $\kappa>1$ and $\eta,\varsigma\in (0,1)$, $\Q$-almost surely there exists $r_{\kappa,\eta,\varsigma}(\w)>0$ such that for all $r\ge r_{\kappa,\eta,\varsigma}(\w)$, there is $z:=z(r,\w)\in \R^d$ with $|z|\le M_{\kappa,\eta}(r)$,
\begin{equation}\label{e:eigen1}\lambda_{V^\w,B(z,r)}\le (1+\varsigma)\lambda_1^{(\alpha)}(B(0,1)) r^{-\alpha},\end{equation} where
 $$M_{\kappa,\eta}(r)=r^{-\kappa}\exp\left(\frac{w_d\rho}{d}((1+2\eta)r)^d\right),$$ and $\lambda_1^{(\alpha)}(B(0,1))$ is the principle Dirichlet eigenvalue for the symmetric $\alpha$-stable process killed upon exiting $B(0,1)$ and with the exponent $\psi^{(\alpha)}(\xi)$ given in \eqref{e:a1-1}.
 \item[(ii)] In the heavy tailed case {\rm(H)}, for any $l>1$ large enough, $\kappa>1$ and $\varsigma\in (0,1)$, $\Q$-almost surely there exists $r_{l,\kappa,\varsigma}(\w)>0$ such that for all $r\ge r_{l,\kappa,\varsigma}(\w)$, there is $z:=z(r,\w)\in \R^d$ with $|z|\le M_\kappa(r)$,
\begin{equation}\label{e:eigen2}\lambda_{V^\w,B(z,lr^{\beta/\alpha})}\le(1+\zeta) q_1r^{-\beta},\end{equation}
where $$M_\kappa(r)=r^{-\kappa} e^{r^d}$$ and \begin{equation}\label{q_1}q_1=\frac{d}{d+\beta}\left(\frac{\beta}{d(d+\beta)}\right)^{\beta/d} \left[\rho w_d \Gamma\left(\frac{\beta}{d+\beta}\right)\right]^{(d+\beta)/d} K.\end{equation}   \end{itemize}
 \end{proposition}
\begin{proof} The proof of the assertion (ii) is a little more delicate, and we postpone it into the appendix. Here we only give the proof of the assertion (i). Note that the argument for the assertion (i) with some modifications works for the critical case; see Proposition \ref{L:lem-l4}.
Fix $\kappa>1$ and $\eta\in (0,1)$, and set
$I_r:=( (2(1+\eta)r)\Z^d)\cap \{z\in \R^d: |z|\le M_{\kappa,\eta}(r)\}$ for any $r>0$. Define $\varphi_0(r)=\sup_{|x|\ge r}\varphi(x)$ for all $r\ge0$ and $\varphi_0(x)=\varphi_0(|x|)$ for $x\in \R^d$. It is clear that $\varphi(x)\le \varphi_0(x)$ for all $x\in \R^d$, and $\varphi_0(r)$ is a  decreasing function on $[0,\infty)$ such that \begin{equation}\label{e:ppp}\lim_{r\to \infty}\varphi_0(r) r^{d+\alpha}=0.\end{equation}   For any $z\in I_r$ and $\varepsilon\in (0,1)$, define
\begin{align*}F_r(z)=&\left\{\hbox{the ball } B(z,(1+\eta)r) {\hbox { at least contains one Poission point}}\right\},\\
G_r(z)=&\left\{\sup_{y\in B(z,r)}\sum_{\omega_i\notin B(z,(1+\eta)r)}\varphi_0(y-\omega_i)\ge   \varepsilon r^{-\alpha} \right\}.\end{align*} We will estimate
 $\Q(\cap_{ z\in I_r}(F_r(z)\cup G_r(z))).$

Note that $\{F_r(z)\}_{z\in I_r}$ are i.i.d., and that
$\Q(F_r(0))=1-e^{-w_d \rho ((1+\eta)r)^d}.$ Hence, there is $r_0(\kappa,\eta)>0$ such that for all $r\ge r_0(\kappa,\eta)$,
\begin{align*}\Q(\cap_{ z\in I_r} F_r(z))&\le (1-e^{-w_d \rho ((1+\eta)r)^d})^{\frac{1}{2}\big(\frac{M_{\kappa, \eta}(r)}{(1+\eta)r}\big)^d}\\
&\le \exp\left(-\frac{1}{2}e^{-w_d \rho ((1+\eta)r)^d}\left(\frac{M_{\kappa, \eta}(r)}{(1+\eta)r}\right)^d \right)\\
&\le \exp\left(-2^{-1}(1+\eta)^{-d} r^{-d(1+\kappa)}e^{-w_d \rho ((1+\eta)r)^d} e^{w_d\rho((1+2\eta) r)^d}   \right) \\
&\le  \exp(-r^d ),\end{align*} where in the second inequality we used the fact that $1-x\le e^{-x}$ for all $x>0$.

On the other hand, for $y\in B(0,r)$ and $\omega_i\notin B(0,(1+\eta)r)$, $|y-\omega_i|\ge \eta |\omega_i|/(1+\eta)$. By the fact that $\varphi_0(x)=\varphi_0(|x|)$ and  the deceasing property of $\varphi_0(r)$,
$$\sup_{y\in B(0,r)}\sum_{\omega_i\notin B(0,(1+\eta)r)}\varphi_0(y-\omega_i)\le \sum_{\omega_i\notin B(0,(1+\eta)r)}\varphi_0(\eta|\omega_i|/(1+\eta)).$$ Hence,
\begin{align*} &\Q\left[\exp\left(\frac{1}{\varphi_0(\eta r)}\sup_{y\in B(0,r)}\sum_{\omega_i\notin B(0,(1+\eta)r)}\varphi_0(y-\omega_i) \right)\right]\\
&\le \Q\left[\exp\left(\frac{1}{\varphi_0(\eta r)} \sum_{\omega_i\notin B(0,(1+\eta)r)}\varphi_0(\eta |\omega_i|/(1+\eta)) \right)\right]\\
&= \exp\left(\rho\int_{\R^d\backslash B(0,(1+\eta)r)}\left(e^{\varphi_0(\eta r)^{-1}\varphi_0(\eta|z|/(1+\eta))}-1\right)\,dz   \right)\\
&\le \exp\left( e \rho(1+\eta)^d \int_{\R^d\backslash B(0, r)}\frac{\varphi_0(\eta z)}{\varphi_0(\eta r)}\,dz\right),\end{align*} where in the last inequality we used the fact that $e^x-1\le e x$ for all $x\in (0,1]$.
By \eqref{e:ppp}, for any $\varepsilon\in (0,1)$ there is a constant $r_1(\eta, \varepsilon)\ge r_0(\kappa,\eta)$ such that for all $r\ge r_1(\eta, \varepsilon)$,
\begin{equation}\label{e:note000}\varphi_0(\eta r)\le \varepsilon^2 (\eta r)^{-d-\alpha}\end{equation} and so
$$\Q\left[\exp\left(\frac{1}{\varphi_0(\eta r)}\sup_{y\in B(0,r)}\sum_{\omega_i\notin B(0,(1+\eta)r)}\varphi_0(y-\omega_i) \right)\right]\le \exp\left[\frac{ c_1 \varepsilon ^2 r^{-\alpha}}{\varphi_0(\eta r)}\right],$$ where $c_1:=c_1(\eta)>0$ depends on $\eta$ but is independent of $\varepsilon$ and $r$.

Below we let $\varepsilon\in (0,1\wedge(1/(2c_1))$.
Hence, according to the Markov inequality and \eqref{e:note000}, for $r$ large enough,
\begin{align*} \Q(G_r(0))\le & \Q\left[\exp\left(\frac{1}{\varphi_0(\eta r)}\sup_{y\in B(0,r)}\sum_{\omega_i\notin B(0,(1+\eta)r)}\varphi_0(y-\omega_i) \right)\ge \exp\left( \frac{\varepsilon r^{-\alpha} }{\varphi_0(\eta r)}\right)\right]\\
\le&  \exp\left[\frac{ c_1\varepsilon^2 r^{-\alpha}}{\varphi_0(\eta r)}- \frac{ \varepsilon r^{-\alpha} }{\varphi_0(\eta r)} \right]\le \exp\left[- \frac{ \varepsilon r^{-\alpha} }{2\varphi_0(\eta r)} \right]\le \exp(-\eta^{d+\alpha}r^d/(2\varepsilon)). \end{align*}
Since $\{G_r(z)\}_{z\in I_r}$ have the same distribution (but are not independent with each other), we find that for any $0<\varepsilon\le \varepsilon_0:=\min\{1,1/(2c_1), \eta^{d+\alpha}/(4w_d\rho(1+2\eta)^d)\}$ and
$r$ large enough,
\begin{align*}\Q(\cup_{ z\in I_r}G_r(z))\le &2\left(\frac{M_{\kappa,\eta}(r)}{(1+\eta)r}\right)^d \exp(-\eta^{d+\alpha}r^d/(2\varepsilon)) \\
=&2(1+\eta)^{-d}r^{-(1+\kappa)d}\exp\left(w_d\rho (1+2\eta)^dr^d-\eta^{d+\alpha}r^d/(2\varepsilon)\right)\\
\le & 2(1+\eta)^{-d}r^{-(1+\kappa)d}\exp\left(-\eta^{d+\alpha}r^d/(4\varepsilon)\right).\end{align*}

Combining with both estimates above, we find that for any $\varepsilon\in (0,\varepsilon_0]$ and any $r$ large enough,
\begin{align*}\Q(\cap_{ z\in I_r}(F_r(z)\cup G_r(z))\le &\Q(\cap_{ z\in I_r}F_r(z))+\Q(\cup_{ z\in I_r}G_r(z))\\
\le&\exp(-r^d )+ 2(1+\eta)^{-d}r^{-(1+\kappa)d}\exp\left(-\eta^{d+\alpha}r^d/(4\varepsilon)\right)\\
\le& c_2\exp(- c_3r^d), \end{align*} where $c_2,c_3>0$ (which depend on $\eta,\kappa,\varepsilon$).
The Borel-Cantelli lemma tells us that $\Q$-almost surely there exists $r_{\kappa, \eta,\varepsilon}(\w)>0$ such that for all $r\ge r_{\kappa, \eta,\varepsilon}(\w)$, there is  $z:=z(r,\w)\in \R^d$ with $|z|\le M_{\kappa,\eta}(r) $ so that both $F_r(z)$ and $G_r(z)$ fail to hold.

Below, we fix this $z$ for all $r\ge r_{\kappa, \eta,\varepsilon}(\w)$.  Since $G_r(z)$ fails to occur,
$$\sup_{y\in B(z,r)}\sum_{\w_i\notin B(z,(1+\eta)r)}\varphi_0(y-\w_i)\le \varepsilon r^{-\alpha} $$ and so, also thanks to $\varphi(x)\le \varphi_0(x)$,
$$ \lambda_{V^\w,B(z,r)}\le \lambda_{\tilde V^\w,B(z,r)}+\varepsilon r^{-\alpha} ,$$ where
$$\tilde V^\w(x)=\sum_{\w_i\in B(z,(1+\eta)r)}\varphi_0(x-\w_i).$$ On the other hand, because $F_r(z)$ does not happen, $\tilde V^\w(x)=0$ for all $x\in \R^d$, and so
$$  \lambda_{\tilde V^\w,B(z,r)}=\lambda_1 (B(z,r)).$$
Therefore, for any $\zeta\in (0,1)$ and $r\ge r_{\kappa, \eta,\varepsilon}(\w)$ large enough,
\begin{align*}\lambda_{V^\w,B(z,r)}\le &\lambda_1 (B(z,r))+\varepsilon r^{-\alpha} \\
= &\lambda_1 (B(0,r))+\varepsilon r^{-\alpha} \le (1+\zeta/2)r^{-\alpha}\lambda_1^{(\alpha)}(B(0,1))+\varepsilon r^{-\alpha}.\end{align*}   where in the last inequality we used Lemma \ref{L:lem-l33} below. The proof is completed by taking $\varepsilon\le \min\{\varepsilon_0,\zeta\lambda_1^{(\alpha)}(B(0,1))/2\}$.\end{proof}

The following was proved in \cite[Proposition 5.1]{KP}.
\begin{lemma}\label{L:lem-l33} Let $Z$ be a symmetric L\'evy process satisfying \eqref{e:a1}, and $\lambda_{1}(D)$ be the principle Dirichlet eigenvalue of the process $Z$ killed when exiting the domain $D$. Then, for any fixed $\zeta>0$, there is $r_0:=r_0(\zeta)>0$ so that for all $r\ge r_0$,
 $$\lambda_{1}(B(0,r))\le (1+\zeta) r^{-\alpha} \lambda_1^{(\alpha)}(B(0,1)),$$ where $\lambda_1^{(\alpha)}(B(0,1))$ is the principle Dirichlet eigenvalue of the symmetric $\alpha$-stable process $Z^{(\alpha)}$ with characteristic exponent $\psi^{(\alpha)}(\xi)$ given in \eqref{e:a1-1} and killed when exiting $B(0,1)$.\end{lemma}

\begin{remark}\label{r:3.6}For our use later, we will consider $a^{-2} V^\w$ with $a>1$ instead of $V^\w$. Here, we note that the following conclusions for the potential $a^{-2}V^\w$, which immediately follow from the proof of Proposition \ref{L:lem-l3}.
 \begin{itemize}
 \item[(i)] In the light tailed case (L), for any $a>1$, \eqref{e:eigen1} holds for $\lambda_{a^{-2}V^\w,B(z,r)}$ in place of $\lambda_{V^\w,B(z,r)}$ with some $\kappa>1$ and $\eta,\varsigma\in (0,1)$ (independent of $a$) and for all $r\ge r_{\kappa, \eta, \varsigma,a}(\w)$ (which depends on $a$);
     \item[(ii)]In the heavy tailed case (H),  for any $a>1$, \eqref{e:eigen2} holds for $\lambda_{a^{-2}V^\w,B(z,r)}$ in place of $\lambda_{V^\w,B(z,r)}$ with $\kappa, l>1$ and $\varsigma\in (0,1)$ (all of which are independent of $a$),
$$q^*_1=a^{-2}q_1=\frac{d}{d+\beta}\left(\frac{\beta}{d(d+\beta)}\right)^{\beta/d} \left[\rho w_d \Gamma\left(\frac{\beta}{d+\beta}\right)\right]^{(d+\beta)/d} a^{-2 }K$$ (in place of $q_1$), and for all $r\ge r_{l,\kappa, \varsigma,a}(\w)$ (which depends on $a$).\end{itemize} \end{remark}

\subsection{Refinement of  quenched estimates of $u^{\omega}(t,0)$}\label{sub-3.3}

 \begin{theorem}\label{T:upper} Assume that for any $t, R\ge 1$ with $R\ge \phi(t)$,
\begin{equation}\label{e:upp}\Pp_0(\tau_{B(0,R)}\le t)\le \Phi(t,R),\end{equation} where $\phi(t)$ is an increasing function on $[1,\infty)$ with $\phi(1)\ge1$, and $\Phi(v_1,v_2)$ is a non-negative function defined on $[1,\infty)^2$ such that $v_1\mapsto \Phi(v_1,v_2)$ is increasing for fixed $v_2$ and $v_2\mapsto \Phi(v_1,v_2)$ is deceasing for fixed $v_1$. Let $\kappa_0$ be the constant defined in \eqref{e:con-}. Then,
\begin{itemize}
\item[(i)]  In the light tailed case {\rm(L)}, for any $\varepsilon>0$, $\Q$-almost surely there is $R_\varepsilon(\w)\ge1$ so that for any $R\ge \max\{R_\varepsilon(\w),\phi(t)\}$ and $t\ge 1$,
$$u^{\w}(t,0)\le  \Phi(t,R)+ C(\varepsilon) R^{d/2} \exp\left(-t(1-2\varepsilon)\left(\frac{ k_0}{d\log R}\right)^{\alpha/d}\right),$$ where $C(\varepsilon)$ is independent of $R$ and $t$.

\item[(ii)]  In the heavy tailed case {\rm(H)}, for any $\varepsilon>0$ and $a>1$, $\Q$-almost surely there is $R_{\varepsilon,a}(\w)\ge1$ so that for any $R\ge \max\{R_{\varepsilon,a}(\w),\phi(t)\}$ and $t\ge 1$,
$$u^{\w}(t,0)\le  \Phi(t,R)+
C(\varepsilon,a) R^{d/a} \exp\left(-t(1-2\varepsilon)\left(\frac{ k_0}{d\log R}\right)^{\beta/d}\right),$$ where   $C(\varepsilon,a)>0$ is independent of $R$ and $t$.
\end{itemize}
    \end{theorem}

 \begin{proof}  We first consider the light tailed case. According to Proposition \ref{L:lim-u} with $\delta$ small enough and \eqref{e:lim-u}, for any $\varepsilon>0$, $\Q$-almost surely  there is $R_\varepsilon(\w)\ge1$ so that for any $t\ge 1$ and $R\ge \max\{R_\varepsilon(\w), \phi(t)\}$,
$$u^{\w}(t,0)\le\Phi(t,R)+C_1(\varepsilon) R^{d/2} \exp\left(-t(1-2\varepsilon)\left(\frac{ k_0}{d\log R}\right)^{\alpha/d}\right).$$

In the heavy tailed case, we note that, from the argument for \eqref{e:lim-u}, for all  $\varepsilon\in (0,1)$ and $a>1$, $\Q$-almost surely  there is $R_{\varepsilon,a}(\w)\ge1$ such that for every $R\ge R_{\varepsilon,a}(\w)$,
$$ a^{-1}\lambda_{aV^\w,B(0,R)}\ge (1-\varepsilon) \left(\frac{k_0}{d \log R}\right)^{ \beta /d},$$ where the right hand side of the inequality above is independent of $a$. With this, we can obtain the desired assertion by following the arguments in the light tailed case.   \end{proof}

 \begin{theorem}\label{T-low} Assume that for any $\delta\in (0,1/2)$ and $r\ge1$
\begin{equation}\label{e:low}\inf_{z\in B(0,r)} p^{B(0,2r)}(\delta,0,z)\ge \Psi_\delta(r) ,\end{equation} where $\Psi_\delta(r)$ is a non-negative deceasing function on $[1,\infty)$. Then,
 \begin{itemize}
 \item[(i)] In the light tailed case {\rm(L)}, for  any $\delta\in (0,1/2)$, $\kappa>1$, $a>1$, $\eta,\varsigma\in (0,1)$, $\Q$-almost surely  there is $R_{\kappa,a,\eta,\varsigma}(\w)\ge1$ so that for any $R\ge R_{\kappa,a,\eta,\varsigma}(\w)$ and $t\ge 1$,
\begin{align*}u^{\w}(t,0)\ge & C(\kappa,\delta,\eta,a)  M_{\kappa,\eta}(R)^{-4\delta d}[\Psi_\delta(2 M_{\kappa,\eta}(R))]^a  \exp\left(-a^2(1+\varsigma)\lambda_1^{(\alpha)}(B(0,1))tR^{-\alpha} \right),\end{align*} where  $$M_{\kappa,\eta}(R)=R^{-\kappa}\exp\left(\frac{w_d\rho}{d}((1+2\eta)R)^d\right),$$ and $\lambda_1^{(\alpha)}(B(0,1))$ is the principle Dirichlet eigenvalue for the symmetric $\alpha$-stable process killed upon exiting $B(0,1)$ and with the exponent $\psi^{(\alpha)}(\xi)$ given in \eqref{e:a1-1}.

\item[(ii)] In the heavy tailed case {\rm(H)}, for any $\delta\in (0,1/2)$, $ \kappa>1$ large enough, $a>1$ and $\varsigma\in (0,1)$, $\Q$-almost surely there is $R_{ \kappa,a,\varsigma}(\w)\ge1$ so that for any $R\ge R_{ \kappa,a, \varsigma}(\w)$ and $t\ge 1$,
\begin{align*}u^{\w}(t,0)\ge & C(\kappa,\delta,\varsigma,a) M_\kappa(R)^{-4\delta d}[\Psi_\delta(2M_\kappa(R))]^a  \exp\left(-  (1+\varsigma)q_1tR^{-\beta} \right),\end{align*} where  $$M_\kappa(R)=R^{-\kappa}\exp({R^d}) $$ and $q_1$ is given by \eqref{q_1}.
\end{itemize}
    \end{theorem}

\begin{proof} We only prove the assertion (i), since the assertion (ii) can be verified similarly by applying Proposition \ref{L:lem-l3}(ii) and Remark \ref{r:3.6}(ii) instead of Proposition \ref{L:lem-l3}(i) and Remark \ref{r:3.6}(i), respectively.

For any $a>1$, $\kappa>1$ and $\eta, \varsigma\in (0,1)$, let $D=B(0,2M_{\kappa, \eta}(r))$ and $D_1=B(z,(1+\eta)r)$ for $r\ge r_{\kappa,\eta,\varsigma,a}(\w)$, where $r_{\kappa,\eta,\varsigma,a}(\w)$, $M_{\kappa,\eta}(r)$ and  $z:=z(r,\w)$ are given in Proposition \ref{L:lem-l3}(i) and Remark \ref{r:3.6}(i). Since $|z|\le M_{\kappa,\eta}(r)$, $D_1\subset D$ for $r$ large enough. Then, according to Propositions \ref{P:prol} and \ref{L:lem-l3}(i) as well as Remark \ref{r:3.6}(i), for any $\delta\in (0,1/2)$, $t\ge1$ and $r$ large enough,
\begin{align*} u^\w(t,0)&\ge p(\delta,0)^ap({t-\delta},0)^{-a^2/b}(w_d((1+\eta)r)^d)^{-2a^2/b} \exp\left(-3d\delta\log (2M_{\kappa,\eta}(r))\right)\\
&\quad\times [\Psi_\delta( 2M_{\kappa,\eta}(r))]^a\exp\left(-a^2(1+\varsigma)tr^{-\alpha} \lambda_1^{(\alpha)}(B(0,1))\right)\\
&\ge C_1(\delta,\eta,a ) r^{-2a^2d/b}  M_{\kappa,\eta}(r)^{-3\delta d}\\
&\quad \times [\Psi_\delta( 2M_{\kappa,\eta}(r))]^a  \exp\left(-a^2(1+\varsigma)tr^{-\alpha} \lambda_1^{(\alpha)}(B(0,1))\right)\\
&\ge C_2(\kappa,\delta,\eta,a )    M_{\kappa,\eta}(r)^{-4\delta d} [\Psi_\delta(2 M_{\kappa,\eta}(r))]^a  \exp\left(-a^2(1+\varsigma)tr^{-\alpha} \lambda_1^{(\alpha)}(B(0,1))\right), \end{align*} where in the first inequality $b>1$ such that $1/a+1/b=1$ and we used \eqref{e:upper-v} and \eqref{e:low}, and the second inequality follows from the fact that for all $t\ge1$ and $\delta\in (0,1/2)$,
$p(t-\delta,0)\le p(\delta,0),$ due to the deceasing property of the function $t\mapsto p(t,0)$.  The proof is finished. \end{proof}

\section{Examples}\label{section4}
In this section, we will present the proofs of Theorems \ref{Th-1} and \ref{Thm1.2}. Note that, both symmetric L\'evy processes in these two examples are rotationally invariant  and satisfy the assumptions in Subsection \ref{section2.1}. Meanwhile, the shape function $\varphi(x)=1\wedge |x|^{-d-\beta}$ with $\beta\in (0,\infty]$ fulfills the assumptions in Subsection \ref{section2.2} as well.

\subsection{Rotationally invariant symmetric $\alpha$-stable processes}

\begin{proof}[Proof of Theorem $\ref{Th-1}$]
For rotationally symmetric $\alpha$-stable process $Z$ with $\alpha\in (0,2)$, $\psi(\xi)=c_0|\xi|^\alpha$ for some $c_0>0$, and so \eqref{e:a1} holds with $\psi^{(\alpha)}(\xi)=\psi(\xi)$. Thus, for given shape function $\varphi(x)=1\wedge|x|^{-d-\beta}$ with $\beta\in (0,\infty]$, the light tailed case (resp.\ the heavy tailed case) corresponds to $\beta>\alpha$ (resp. $\beta\in(0,\alpha)$).
Furthermore, it is well known that, for rotationally symmetric $\alpha$-stable process $Z$, \eqref{e:upp} holds with
$$\Phi(t,r)=C^*t r^{-\alpha}$$ and $\phi(t)=t^{1/\alpha}$, and
\eqref{e:low} holds with
$$\Psi_\delta (r)\ge \frac{C_* \delta }{r^{d+\alpha}};$$ see \cite{CK,CKS}.

(i)  We first consider $\beta>\alpha$, which is refereed to the light tailed case. According to Theorem \ref{T:upper}(i),  for any $\varepsilon>0$, $\Q$-almost surely there is $R_\varepsilon(\w)\ge1$ so that for any $t\ge 1$ and  $R\ge \max\{R_\varepsilon(\w), t^{1/\alpha}\}$,
$$u^\w(t,0)\le\frac{ C_1t}{R^\alpha}+C_2(\varepsilon) R^{d/2} \exp\left(-t(1-2\varepsilon)\left(\frac{ k_0}{d\log R}\right)^{\alpha/d}\right),$$ where $C_2(\varepsilon)>0$ is a constant independent of $R$ and $t$, and $k_0=\rho w_d[\lambda_1^{(\alpha)}(B(0,1))]^{d/\alpha} $ with $\lambda_1^{(\alpha)}(B(0,1))$ being the principle Dirichlet eigenvalue for the rotationally  symmetric $\alpha$-stable process $Z$ killed upon exiting $B(0,1)$.
Letting
$$R=\exp\left((1-2\varepsilon)^{d/(\alpha+d)}(\alpha+d/2)^{-d/(\alpha+d)} \left(\frac{k_0}{d}\right)^{\alpha/(d+\alpha)} t^{d/(d+\alpha)}\right)$$ for $t$ large enough, we arrive at the desired upper bound by letting $\varepsilon \to0$.

On the other hand, by Theorem \ref{T-low}(i), for any $\kappa,a>1$, $\delta\in (0,1/2)$ and $ \eta,\varsigma\in (0,1)$, $\Q$-almost surely there is $R_{\kappa, a,\eta,\varsigma}(\w)\ge1$ so that for any $R\ge R_{\kappa, a,\eta,\varsigma}(\w)$ and $t\ge 1$,
$$u^\w(t,0)\ge C  R^{(4\delta d+(d+\alpha)a)\kappa}\exp\left(-A R^d-BtR^{-\alpha}\right),$$ where
$$A=\frac{w_d\rho}{d}(1+2\eta)^d[a(d+\alpha)+4\delta d],\quad B =a^2(1+\varsigma)\lambda_1^{(\alpha)}(B(0,1))$$ and $C>0$ is a constant independent of $t$ and $R$.
Letting $$R=\left(\frac{\alpha B}{d A}\right)^{1/(d+\alpha)} t^{1/(d+\alpha)}$$ for $t$ large enough, we prove the lower bound by taking $\delta,\eta,\varsigma  \to 0$ and $a\to 1$.

(ii) For heavy tailed case (i.e., $\beta\in (0,\alpha)$), it follows from Theorem \ref{T:upper}(ii) that for all $a>1$ and $\varepsilon\in (0,1)$, $\Q$-almost surely  there is $R_{a,\varepsilon}(\w)\ge1$ such that
for any $t\ge 1$ and  $R\ge \max\{R_{a,\varepsilon}(\w), t^{1/\alpha}\}$,
$$u^\w(t,0)\le\frac{ C_1t}{R^\alpha}+C_2(a,\varepsilon) R^{d/a} \exp\left(-t(1-2\varepsilon)\left(\frac{ k_0}{d\log R}\right)^{\beta/d}\right),$$ where $k_0$ is given by \eqref{e:con-} in case {\rm(H)}. Then, choosing $$R=\exp\left((1-2\varepsilon)^{d/(\beta+d)}(\alpha+d/a)^{-d/(\beta+d)} \left(\frac{k_0}{d}\right)^{\beta/(d+\beta)} t^{d/(d+\beta)}\right)$$ for $t$ large enough, we arrive at the upper bound by letting $\varepsilon\to0$ and $a\to\infty$.

Due to Theorem \ref{T-low}(ii), for any $\kappa,a>1$, $\delta\in(0,1/2)$  and $ \varsigma\in (0,1)$, $\Q$-almost surely there is $R_{\kappa,  a,\varsigma}(\w)\ge1$ so that for any  $R\ge R_{\kappa, a, \varsigma}(\w)$ and $t\ge 1$,
$$u^\w(t,0)\ge   C  R^{(4\delta d+(d+\alpha)a)\kappa}\exp\left(-A R^d-BtR^{-\beta}\right),$$ where
$$A=a(d+\alpha)+4\delta d ,\quad B =   (1+\varsigma)q_1.$$
Letting $$r=\left(\frac{\beta B}{d A}\right)^{1/(d+\beta)} t^{1/(d+\beta)}$$ for $t$ large enough, we prove the lower bound by taking $\varsigma,\delta \to 0$ and $a\to 1$.
\end{proof}

\begin{remark} \begin{itemize}
\item[(i)] We used two different ways to estimate $I_1$ in the proof of Proposition \ref{L:lim-u}, which yield two different quenched upper bounds for $u^\w(t,0)$ in Theorem \ref{T:upper}. For this example, if we follow the argument for light tailed case (i.e., $\beta\in (\alpha,\infty]$)
to deal with the heavy tailed case (i.e., $\beta\in (0,\alpha)$), then we can only obtain that when $\beta\in (0,\alpha)$, $\Q$-almost surely for all $x\in \R^d$,
$$\limsup_{t\to\infty} \frac{u^\w(t,x)}{t^{d/(d+\beta)}}\le -\frac{\alpha}{(\alpha+d/2)^{d/(d+\beta)}}A_2,$$
which is weaker than the desired assertion for the upper bound in Theorem \ref{Th-1}(ii).
\item[(ii)] As mentioned in Section \ref{section1}, rate functions for the quenched and annealed asymptoticses of $u^\w(t,x)$ for rotationally symmetric $\alpha$-stable processes are same. However, we can not obtain that the associated liminf and limsup constants agree for the quenched asymptotics. The reason why our argument can not yield precise results is due to that both estimates \eqref{e:upp} and \eqref{e:low} are of the polynomial-form for symmetric $\alpha$-stable processes.
\end{itemize} \end{remark}

\subsection{Rotationally symmetric processes with large jumps of exponential decay}

\begin{proof}[Proof of Theorem $\ref{Thm1.2}$]For rotationally symmetric pure jump L\'evy process $Z$ with L\'evy measure $\nu$ given in Theorem \ref{Thm1.2}, by \cite[Proposition 5.2(i)]{KP}, \eqref{e:a1} holds with $\alpha=2$ and  $$a_{ij}=\int_{\R^d\backslash \{0\}}z_iz_j\,\nu(dz),\quad 1\le i,j\le d.$$
 Thus, for given shape function $\varphi(x)=1\wedge|x|^{-d-\beta}$ with $\beta\in (0,\infty]$, the light tailed case (resp.\ the heavy tailed case) corresponds to $\beta>2$ (resp. $\beta\in(0,2)$).

Furthermore, according to \cite[Theorems 1.2 and 1.4]{CKK}, for any $t\ge1$ and $x\in \R^d$ with $|x|\ge 2 t^{1/((2-\theta)\vee 1)}$, the transition density function $p(t,x)$ of the process $X$ satisfies that
$$p(t,x)\le c_1\exp\left(-c_2|x|^{\theta\wedge1} \left(\log \frac{|x|}{t}\right)^{(\theta-1)^+/\theta}\right).$$
This along with Lemma \ref{L:3.2} below yields that
\eqref{e:upp} holds with
$$\Phi(t,r)=c_3 \exp(-c_4 r^{\theta \wedge1})$$ and $\phi(t)=2 t $. On the other hand, by \cite[Theorem 1.1]{PP}, we know that
\eqref{e:low} holds with
$$\Psi_\delta (r)\ge c_5 \exp(-c_6 r^{\theta \wedge 1} \left(\log  r \right)^{(\theta-1)^+/\theta}).$$

For simplicity, we only prove the light tailed case (i.e., $\beta\in (2,\infty)$), since the heavy tailed case can be treated similarly. First,  by Theorem \ref{T:upper}(i),  for any $\varepsilon>0$, $\Q$-almost surely there is $R_\varepsilon(\w)\ge1$ so that for any $t\ge 1$ and  $R\ge \max\{R_\varepsilon(\w), 2 t \}$,
$$u^\w(t,0)\le c_3 \exp(-c_4 R^{\theta \wedge1})+C_1(\varepsilon) R^{d/2} \exp\left(-t(1-2\varepsilon)\left(\frac{ k_0}{d\log R}\right)^{2/d}\right),$$ where $C_1(\varepsilon)>0$ is a constant independent of $R$ and $t$, and
$k_0=\rho w_d[\lambda_1^{(2)}(B(0,1))]^{d/2} $ with $\lambda_1^{(2)}(B(0,1))$ being the principle Dirichlet eigenvalue for Brownian motion killed upon exiting $B(0,1)$ and with the covariance matrix $(a_{ij})_{1\le i,j\le d}$ above.
Letting
$R=C t^{1/(1\wedge \theta)}$ for large $C$ and $t$, we prove the desired upper bound by taking $\varepsilon \to 0.$

On the other hand, according to Theorem \ref{T-low}(i), for any $\kappa,a>1$ and $ \eta,\varsigma\in (0,1)$, $\Q$-almost surely there is $R_{\kappa, a,\eta,\varsigma}(\w)\ge1$ so that for any $R\ge R_{\kappa, a,\eta,\varsigma}(\w)$ and $t\ge 1$,
\begin{align*}u^\w(t,0)\ge &C_2\exp\left(-C_3 (M_{\kappa,\eta}(R))^{a(\theta\wedge1)}(\log M_{\kappa,\eta}(R))^{(\theta-1)^+/\theta}\right)\\
&\times \exp\left(-a^2(1+\varsigma)\lambda_1^{(2)}(B(0,1)) tR^{-2}\right)\\
\ge& C_4\exp\left[-C_5R^{-\kappa a(\theta\wedge1)+d(\theta-1)^+/\theta}\exp\left(a(\theta\wedge1) \frac{w_d\rho}{d}((1+2\eta)R)^d\right) \right]\\
&\times \exp\left(-a^2(1+\varsigma)\lambda_1^{(2)}(B(0,1)) tR^{-2}\right).\end{align*}
Choosing $\kappa$ large enough and $$R= \frac{1}{1+2\eta}\left( \frac{d}{a(1\wedge \theta)w_d \rho} \right)^{1/d} (\log t)^{1/d}$$ for $t$ large enough, we prove the lower bound by taking $\eta,\varsigma  \to 0$ and $a\to 1$.
\end{proof}

\begin{lemma}\label{L:3.2}For any L\'evy process $Z$, it holds for all $t,R>0$  that
$$ \Pp_0(\tau_{B(0,R)}\le t)\le 2\sup_{s\in [t,2t]}\Pp_0(|Z_s|\ge R/2).$$  \end{lemma}
\begin{proof} For any $t, R>0$,
\begin{align*}\Pp_0(\tau_{B(0,R)}\le t)&=\Pp_0\left(\max_{s\in (0,t]}|Z_s|\ge R\right)\\
&=\Pp_0\left(\max_{s\in (0,t]}|Z_s|\ge R, |Z_{2t}|\ge R/2\right)+\Pp_0\left(\max_{s\in (0,t]}|Z_s|\ge R, |Z_{2t}|\le R/2\right)\\
&\le \Pp_0( |Z_{2t}|\ge R/2) + \Ee_0\left(\I_{\{\tau_{B(0,R)}\le t\}}\Pp_{Z_{\tau_{B(0,R)}}}(|Z_{2t}-Z_{\tau_{B(0,R)}}|\ge R/2)\right)\\
&\le 2\sup_{s\in [t,2t]}\Pp_0(|Z_s|\ge R/2). \end{align*} The proof is complete. \end{proof}
\section{Appendix}
\subsection{Proof of Proposition \ref{L:lem-l3}(ii)}
In this part, we will present the proof of Proposition \ref{L:lem-l3}(ii). The proof mainly follows from the argument in \cite[Section 4.1]{Fuk2}. Note that since the paper \cite{Fuk2} studied  second order asymptotics for Brownian motions in a heavy tailed Poissonian potential, the proof is much more involved. In particular, the argument in \cite[Section 4.1]{Fuk2} only works for part of heavy tailed potentials (i.e., for the shape function $\varphi(x)=1\wedge |x|^{-(d+\beta)}$ with $\beta\in (0,2)$ and $d+\beta\ge 2$). Now, in our setting we can prove Proposition \ref{L:lem-l3}(ii) holds for all heavy tailed potentials, because only the first order asymptotics for the first Dirichlet eigenvalue is concerned here.

To highlight differences from the argument in \cite[Section 4.1]{Fuk2}, we rewrite Proposition \ref{L:lem-l3}(ii) as follows, where notations are the same as those in \cite{Fuk2}.

\begin{proposition}\label{P:5.1} In the heavy tailed case {\rm(H)}, for $M>1$ large enough, any $\kappa>1$ and $\varepsilon>0$, $\Q$-almost surely there exists $t_{M,\kappa,\varepsilon}(\w)>0$ such that, for all $t\ge t_{M,\kappa,\varepsilon}(\w)$ there is $z:=z(t,\w)\in \R^d$ so that $|z|\le t (\log t)^{-\kappa}$ and
$$\lambda_{B(z, M(\log t)^{\beta /(d \alpha)})}\le (1+\varepsilon)\lambda(t),$$ where
$\lambda(t)=q_1(\log t)^{-\beta/d}$, and $q_1$ is given by \eqref{q_1}.
\end{proposition}
In the heavy tailed case, by the continuity of $\varphi$ and \eqref{:h}, for any $\theta>0$, there exists a constant $C(\theta)>0$ such that for all $x\in \R^d$,
$$\varphi(x)\le \varphi_0(x):=(K+\theta)(C(\theta)\wedge |x|^{-d-\beta}).$$ Thus, to consider  upper bounds for the first Dirichlet eigenvalue corresponding to the shape function $\varphi$, it suffices to study that associated with the function $\varphi_0$. For simplicity, in the proof below we just take
$$\varphi_0(x):=1\wedge |x|^{-d-\beta},\quad x\in \R^d,$$ since the argument goes through for $\varphi_0(x)=(K+\theta)(C(\theta)\wedge |x|^{-d-\beta})$ and then the desired assertion follows by letting $\theta$ small enough.

  \ \

Let $N>1/  d $, and $M>1$ large enough. Define
$\Lambda_N(t)=[-(\log t)^N,(\log t)^N]$ and $B_M(t)= B(0, M(\log t)^{\beta/(d \alpha)})$.
 First, we have
 \begin{lemma}\label{L:5.1} For any $\varepsilon>0$, there is a constant $c(\varepsilon)>0$ such that for all $t$ large enough,
 $$\Q\left(\sup_{y\in B_M(t)}\sup_{\w_i\notin \Lambda_N(t)} |y-\w_i|^{-d-\beta}>\varepsilon(\log t)^{-\beta/d}\right)\le \exp\left(-c(\varepsilon) (\log t)^{ d N+\beta (N-1/d)}\right).$$
 \end{lemma}
 \begin{proof} For $t$ large enough, and for $\w_i\notin \Lambda_N(t)$ and $y\in B_M(t)$, by $N>1/d$ and $\beta\in (0,\alpha)$, we have
 $|\w_i-y|\ge |\w_i|/2$, and so
 $$\sup_{y\in B_M(t)}\sum_{\w_i\notin \Lambda_N(t)}|y-\w_i|^{-d-\beta}\le 2^{d+\beta} \sum_{\w_i\notin \Lambda_N(t)} |\w_i|^{-d-\beta}.$$¡¡Note that, since $\{\w_i\}$ are from a realization of a homogeneous Poisson point process on $\R^d$ with parameter $\rho$, for $t$ large enough,
 \begin{align*}&\Ee_\Q\exp\left\{(\log t)^{(d+\beta)N} \sum_{\w_i\notin \Lambda_N(t)}|\w_i|^{-d-\beta}\right\}\\
 &=\exp\left(\rho \int_{\R^d\backslash \Lambda_N(t)}\left(e^{(\log t)^{(d+\beta)N}|z|^{-(d+\beta)}}-1\right)\,dz\right)\\
 &\le \exp\left(\rho e\int_{\R^d\backslash \Lambda_N(t)} (\log t)^{(d+\beta)N}|z|^{-(d+\beta)}\,dz\right) \le \exp(c_1(\log t)^{dN}),
 \end{align*} where in the first inequality we used the fact that $e^x-1
\le e x$ for all $x\in (0,1]$.
 Therefore, by the Markov inequality, for $t$ large enough,
 \begin{align*}& \Q\left(\sup_{y\in B_M(t)}\sup_{\w_i\notin \Lambda_N(t)} |y-\w_i|^{-d-\beta}>\varepsilon(\log t)^{-\beta/d}\right)\\
 &\le  \Q\left(\sum_{\w_i\notin \Lambda_N(t)}|\w_i|^{-d-\beta}\ge 2^{-d-\beta}\varepsilon (\log t)^{-\beta/d}\right)\\
 &\le \exp\left( c_1(\log t)^{dN}-\varepsilon 2^{-d-\beta} (\log t)^{(d+\beta)N-\beta/d}\right)\\
 &\le \exp\left(-c_2 (\log t)^{ d N+\beta (N-1/d)}\right),\end{align*} where in the last inequality we used the fact that $N>1/d$. The proof is complete.
  \end{proof}

  For any $t>0$, define
  $$H(t)=\log \Ee_\Q[\exp(-t V^\w(0))],\quad \rho_0(t)=\left(\frac{(d+\beta)t}{d a_1}\right)^{-(d+\beta)/\beta},$$ where $$a_1=\rho w_d\Gamma\left(\frac{\beta}{d+\beta}\right).$$
In particular,
$$\rho_0(\lambda(t))=\left(\frac{a_1\beta}{d(d+\beta)}\right)^{-(d+\beta)/d} (\log t)^{ (d+\beta)/d}$$ and
\begin{equation}\label{e:llss}H(\rho_0(\lambda(t)))+\lambda(t)\rho_0(\lambda(t))=-d \log t+o(1),\end{equation} where in the latter equality we used the fact that
$$H(t)=-a_1t^{d/(d+\beta)}+O(e^{-t}),\quad t\to \infty;$$ see \cite[Lemma 1]{Fuk2}.
Next, we introduce a transformed measure defined by
$$ { \tilde \Q_t}(d\w) =\left[e^{-H(\rho_0(\lambda(t)))-\rho_0(\lambda(t))V^\w(0)}\right]{ \Q}(d \w),\quad t>0.$$ Then, it follows from \cite[Lemma 7(1)]{Fuk2} that $(\w, \tilde \Q_t)$ is a Poisson point process on $\R^d$ with intensity $\rho e^{-\rho_0(\lambda(t))\varphi_0(z)}\,dz.$ Furthermore, we have

\begin{lemma}\label{L:lem5.1} For $t$ large enough,
$$\sup_{x\in B_M(t)}|\Ee_{\tilde \Q_t} [V^\w(x)]-\lambda(t)|\le o((\log t)^{-\beta/d}).$$   \end{lemma}
\begin{proof} For any $x\in B_M(t)$,
\begin{align*}\Ee_{\tilde \Q_t} [V^\w(x)]&=\rho \int_{\R^d} \varphi_0(x-z)e^{-\rho_0(\lambda(t))\varphi_0(z)}\,dz\\
&=\rho\int_{B_{2M}(t)}\varphi_0(x-z)e^{-\rho_0(\lambda(t))\varphi_0(z)}\,dz+\rho\int_{\R^d\setminus B_{2M}(t)}\varphi_0(x-z)e^{-\rho_0(\lambda(t))\varphi_0(z)}\,dz.\end{align*} It is easy to see that for $t$ large enough
\begin{equation}\label{e:note1}\begin{split} \rho\sup_{x\in B_M(t)}\int_{B_{2M}(t)}\varphi_0(x-z)e^{-\rho_0(\lambda(t))\varphi_0(z)}\,dz
&\le \rho\int_{B_{2M}(t)} e^{-\rho_0(\lambda(t))\varphi_0(z)}\,dz\\
&\le c_1 \exp(-c_2(\log t)^{(d+\beta)(\alpha-\beta)/(d\alpha)}),\end{split}\end{equation} where $c_1,c_2>0$ are independent of $t$ (but depending on $M$). Thus, for $x\in B_M(t)$ and for $t$ large enough,
$$\Ee_{\tilde \Q_t} [V^\w(x)]\le \rho\int_{\R^d\setminus B_{2M}(t)}|x-z|^{-d-\beta}e^{-\rho_0(\lambda(t))|z|^{-d-\beta}}\,dz+¡¡c_1 \exp(-c_2(\log t)^{(d+\beta)(\alpha-\beta)/(d\alpha)}).$$

On the other hand, we can check that
$$\rho\int_{B_{2M}(t)}|z|^{-d-\beta}e^{-\rho_0(\lambda(t))|z|^{-d-\beta}}\,dz\le c_3 \exp(-c_4(\log t)^{(d+\beta)(\alpha-\beta)/(d\alpha)}).$$ Then, for $t>0$ large enough,
\begin{align*}\Ee_{\tilde \Q_t} [V^\w(0)]=& \rho\int_{\R^d}|z|^{-d-\beta}e^{-\rho_0(\lambda(t))|z|^{-d-\beta}}\,dz +c_5 \exp(-c_6(\log t)^{(d+\beta)(\alpha-\beta)/(d\alpha)})\\
=& \lambda (t)+O\left(\exp(-c(\log t)^{(d+\beta)(\alpha-\beta)/(d\alpha)})\right)\end{align*} for some constant $c>0$.

Next, for any $x\in B_M(t)$ and $t$ large enough, by the fact that $\varphi_0(z)=1\wedge |z|^{-(d+\beta)}$ and the mean value theorem,
\begin{align*}|\Ee_{\tilde \Q_t} (V^\w(x)-V^\w(0))|&\le \rho\int_{\R^d\setminus B_{2M}(t)}|\varphi_0(x-z)-\varphi_0(z)|e^{-\rho_0(\lambda(t))\varphi_0(z)}\,dz\\
&\quad+O\left(\exp(-c(\log t)^{(d+\beta)(\alpha-\beta)/(d\alpha)})\right)\\
&\le c_7(\log t)^{\beta/(d \alpha)} \int_{\R^d\setminus B_{2M}(t)}|z|^{-d-\beta-1} e^{-c_8(\log t)^{(d+\beta)/d}|z|^{-d-\beta}}\,dz\\
&\quad+O\left(\exp(-c(\log t)^{(d+\beta)(\alpha-\beta)/(d\alpha)})\right)\\
&\le c_9(\log t)^{-\beta/d-(1-\beta/\alpha)/d},  \end{align*} thanks to $\beta\in (0,\alpha)$ again. This proves the desired assertion.  \end{proof}

Now, we are back to the probability estimate for $V^\w(x)$ under the probability measure $\Q$.
\begin{lemma}\label{L5.2}There is a constant $\delta\in (0,1/2)$ such that for any $\varepsilon, M>0$ there exists $t_{\delta,\varepsilon, M}>0$ such that for all $t\ge t_{\delta,\varepsilon, M}$,
$$\Q\left(\sup_{x\in B_M(t)}|V^\w(x)-\lambda (t)|\le \varepsilon (\log t)^{-\beta/d}\right)\ge c(\delta, \varepsilon, M) t^{-d}\exp((\log t)^{\delta}).$$ \end{lemma}

\begin{proof} For any given $\varepsilon>0$, by Lemma \ref{L:lem5.1}, for $t$ large enough,
 $$\sup_{x\in B_M(t)}|\Ee_{\tilde \Q_t}(V^\w(x)-V^\w(0))|\le \frac{\varepsilon}{4}(\log t)^{-\beta/d}.$$ For any $\gamma\in (1/2,1)$, we further define $$E_1=\left\{ V^\w(0)-\lambda (t) \in \left[(\log t)^{-\beta/d-\gamma},\frac{\varepsilon}{4}(\log t)^{-\beta/d}\right]\right\}$$ and
$$E_2=\left\{\sup_{x\in B_M(t)}|V^\w(x)-V^\w(0)-\Ee_{\tilde \Q_t}(V^\w(x)-V^\w(0))|\ge \frac{\varepsilon}{2}(\log t)^{-\beta/d} \right\}.$$ Then, for $t$ large enough,
$$E_1 E_2^c\subset \left\{\sup_{x\in B_M(t)}|V^\w(x)-\lambda (t)|\le \varepsilon (\log t)^{-\beta/d}\right\}.$$ Hence,
\begin{equation}\label{jjss}\begin{split}&\Q\left(\sup_{x\in B_M(t)}|V^\w(x)-\lambda (t)|\le \varepsilon (\log t)^{-\beta/d}\right)\\
&\ge e^{H(\rho_0(\lambda(t)))}\Ee_{\tilde\Q_t} (e^{\rho_0(\lambda(t))V^\w(0)}\I_{E_1\setminus E_2})\\
&\ge \exp\left(H(\rho_0(\lambda(t)))+\rho_0(\lambda(t))(\lambda(t)+(\log t)^{-\beta/d-\gamma})\right) \tilde \Q_t(E_1\setminus E_2)\\
&\ge \exp\left(-d \log t+\rho_0(\lambda(t))(\log t)^{-\beta/d-\gamma}+o(1)\right) (\tilde \Q_t(E_1)-\tilde \Q_t(E_2))\\
&\ge c_1 t^{-d} \exp(c_2(\log t)^{1-\gamma})(\tilde \Q_t(E_1)-\tilde \Q_t(E_2)),\end{split}\end{equation} where in the third inequality we used \eqref{e:llss}.

As shown in \cite[Lemma 7(iii)]{Fuk2}, $$(\log t)^{(d+2\beta)/(2d)}(V^\w(0)-\lambda (t))$$ under $\tilde \Q_t$ converges in law to a non-degenerate Gaussian random variable as $t\to\infty$. Then,
$$\tilde \Q_t(E_1)=\tilde \Q_t\left((\log t)^{(d+2\beta)/(2d)}(V^\w(0)-\lambda (t))\in  \left[(\log t)^{1/2-\gamma}, \frac{\varepsilon}{4}(\log t)^{1/2}\right]\right)$$ is bounded from below by a positive constant for $t$ large enough, thanks to $\gamma\in (1/2,1)$.

On the other hand, defining
$$\bar\mu^\w_t(dz):=\mu^\w(dz)-\rho e^{-\rho_0(\lambda(t))\varphi_0(z)}\,dz,$$ we write
\begin{align*}&V^\w(x)-V^\w(0)-\Ee_{\tilde \Q_t} (V^\w(x)-V^\w(0))\\
&= \int_{\R^d} \left(\varphi_0(x-z)-\varphi_0(-z)\right)\bar \mu^\w_t(dz)\\
&=\int_{B_{2M(t)}}\left(\varphi_0(x-z)-\varphi_0(-z)\right)\bar \mu^\w_t(dz)+ \int_{\R^d\backslash B_{2M(t)}}\left(\varphi_0(x-z)-\varphi_0(-z)\right)\bar \mu^\w_t(dz).  \end{align*}
Note that, by the fact that $\varphi_0(x)=1\wedge |x|^{-d-\beta}$,
\begin{align*}&\sup_{x\in B_M(t)}\left|\int_{B_{2M(t)}}\left(\varphi_0(x-z)-\varphi_0(-z)\right)\bar \mu^\w_t(dz)\right|\\
&\le \sup_{x\in B_M(t)}\int_{B_{2M(t)}}|\varphi_0(x-z)-\varphi_0(-z)|\,\mu^\w(dz)\\
&\quad +\rho\sup_{x\in B_M(t)}\int_{B_{2M(t)}}|\varphi_0(x-z)-\varphi_0(-z)|e^{-\rho_0(\lambda(t))\varphi_0(z)}\,dz\\
&\le \int_{B_{2M(t)}}\bar \mu^\w_t(dz)+2\rho\int_{B_{2M(t)}} e^{-\rho_0(\lambda(t))\varphi_0(z)}\,dz.  \end{align*}
Hence, according to the second inequality in \eqref{e:note1}, for $t$ large enough,
\begin{align*}&\tilde \Q_t\left(\sup_{x\in B_M(t)}\left|\int_{B_{2M(t)}}\left(\varphi_0(x-z)-\varphi_0(-z)\right)\bar \mu^\w_t(dz)\right|\ge \frac{\varepsilon}{4}(\log t)^{-\beta/d}\right)\\
&\le \tilde \Q_t\left(\int_{B_{2M(t)}}\bar \mu^\w_t(dz)\ge  \frac{\varepsilon}{8}(\log t)^{-\beta/d}\right)\le  \left[\frac{\varepsilon}{8}(\log t)^{-\beta/d}\right]^{-2}\Ee_{\tilde \Q_t} \left[\int_{B_{2M(t)}}\bar \mu^\w_t(dz)\right]^2\\
&=\left[\frac{\varepsilon}{8}(\log t)^{-\beta/d}\right]^{-2}\,\,\rho\int_{B_{2M(t)}} e^{-\rho_0(\lambda(t))\varphi_0(z)}\,dz\\
&\le c_3 \exp(-c_4(\log t)^{(d+\beta)(\alpha-\beta)/(d\alpha)}), \end{align*} where in the second inequality we used the Markov inequality and the equality above follows from the fact that the $\tilde \Q_t$-mean of $\int_{B_{2M(t)}}\bar \mu^\w_t(dz)$ is zero.

Furthermore, according to the mean value theorem, for $t$ large enough,
\begin{align*}&\sup_{x\in B_M(t)}\left|\int_{\R^d\backslash B_{2M(t)}}\left(\varphi_0(x-z)-\varphi_0(-z)\right)\bar \mu^\w_t(dz)\right|\\
&=\sup_{x\in B_M(t)}\left|\int_{\R^d\backslash B_{2M(t)}}\int_0^1   \frac{d}{d\theta}\varphi_0(\theta x-z) \,d\theta \,\bar \mu^\w_t(dz)\right|\\
&\le \int_{\R^d\backslash B_{2M(t)}}\sup_{x\in B_M(t),\theta\in (0,1)} \left|\frac{d}{d\theta}\varphi_0(\theta x-z)\right|\, \mu^\w_t(dz)\\
&\quad+ \rho \int_{\R^d\backslash B_{2M(t)}}\sup_{x\in B_M(t),\theta\in (0,1)} \left|\frac{d}{d\theta}\varphi_0(\theta x-z)\right| e^{-\rho_0(\lambda(t))\varphi_0(z)}\,dz \\
&= \int_{\R^d\backslash B_{2M(t)}}\sup_{x\in B_M(t),\theta\in (0,1)} \left|\frac{d}{d\theta}\varphi_0(\theta x-z)\right|\,\bar  \mu^\w_t(dz)\\
&\quad+2\rho \int_{\R^d\backslash B_{2M(t)}}\sup_{x\in B_M(t),\theta\in (0,1)} \left|\frac{d}{d\theta}\varphi_0(\theta x-z)\right| e^{-\rho_0(\lambda(t))\varphi_0(z)}\,dz \\
&\le c_5(\log t)^{\beta/(d \alpha)} \int_{\R^d\backslash B_{2M(t)}}|z|^{-d-\beta-1}\,\bar \mu^\w_t(dz)\\
&\quad + c_5(\log t)^{\beta/(d \alpha)} \int_{\R^d\backslash B_{2M(t)}}|z|^{-d-\beta-1} e^{-\rho_0(\lambda(t))\varphi_0(z)}\,dz. \end{align*} Note that
$$(\log t)^{\beta/(d \alpha)}\int_{\R^d\backslash B_{2M(t)}}|z|^{-d-\beta-1} e^{-\rho_0(\lambda(t))\varphi_0(z)}\,dz\le c_6(\log t)^{-\beta/d-(1-\beta/\alpha)/d}$$ for $t$ large enough, and that the $\tilde \Q_t$-mean of
$\int_{\R^d\backslash B_{2M(t)}}|z|^{-d-\beta-1}\,\bar \mu^\w_t(dz)$ is zero and the variance of it is bounded above by $c_7(\log t)^{-(d+2\beta+2)/d}$.  Hence, for $t$ large enough, by the Markov inequality,
\begin{align*}&\tilde \Q_t\left(\sup_{x\in B_M(t)}\left|\int_{\R^d\backslash B_{2M(t)}}\left(\varphi_0(x-z)-\varphi_0(-z)\right)\bar \mu^\w_t(dz)\right|\ge \frac{\varepsilon}{4}(\log t)^{-\beta/d}\right)\\
&\le \tilde \Q_t\left( c_5(\log t)^{\beta/(d \alpha)}\int_{\R^d\backslash B_{2M(t)}}|z|^{-d-\beta-1}\,\bar \mu^\w_t(dz)\ge  \frac{\varepsilon}{8}(\log t)^{-\beta/d}\right)\\
&\le c_8 (\log t)^{2(\beta+\beta/\alpha)/d}\,\,\Ee_{\tilde\Q_t}\left[\int_{\R^d\backslash B_{2M(t)}}|z|^{-d-\beta-1}\,\bar \mu^\w_t(dz)\right]^2\\
&\le  c_9 (\log t)^{-1-2(1-\beta/\alpha)/d}. \end{align*}

Combining all the estimates above, we arrive at that $ \tilde \Q_t(E_2)$ tends to zero when $t\to \infty$, and so $\tilde \Q_t(E_1)-\tilde \Q_t(E_2)$ is bounded below by a positive constant for $t$ large enough.
This along with \eqref{jjss} yields the desired assertion.
 \end{proof}

Now, we can present the

\begin{proof}[Proof of Proposition $\ref{P:5.1}$] Fix $\kappa>1$, and set
$I_t:=( (2(\log t)^N)\Z^d)\cap \{z\in \R^d: |z|\le t (\log t)^{-\kappa}\}$ for any $t>0$. For any $z\in I_t$ and $\varepsilon>0$, define
\begin{align*}F_r(z)=&\left\{\sup_{x\in z+B_M(t)}|\tilde V^\w(x)-\lambda(t)|\ge \frac{\varepsilon}{2}(\log t)^{-\beta/d}\right\},\\
G_r(z)=&\left\{ \sup_{x\in z+B_M(t)}\sum_{\w_i\notin z+\Lambda_N(t)}|z-\w_i|^{-d-\beta}\ge \frac{\varepsilon}{4} (\log t)^{-\beta/d}\right\},\end{align*} where
$$\tilde V^\w(x)=\sum_{\w_i\in z+\Lambda_N(t)}|x-\w_i|^{-d-\beta}.$$  We will estimate
 $\Q(\cap_{ z\in I_t}(F_t(z)\cup G_t(z))).$

 Note that $\{G_t(z)\}_{z\in I_t}$ have the same distribution such that for any $z\in I_t$ and $t$ large enough
 $$\Q(G_t(z))=\Q(G_t(0))\le \left(-c_1(\varepsilon) (\log t)^{ d N+\beta (N-1/d)}\right),$$ thanks to Lemma \ref{L:5.1}. On the other hand,  $\{F_t(z)\}_{z\in I_t}$ are i.i.d., and, according to Lemmas \ref{L5.2} and \ref{L:5.1}, for any $z\in I_t$ and $t$ large enough,
 \begin{align*}\Q(F_t(z))&=\Q(F_t(0))\le \Q(F_t(0)\setminus G_t(0))+\Q(G_t(0))\\
 &=\Q\left(\sup_{x\in B_M(t)}|V^\w(x)-\lambda (t)|\ge \frac{\varepsilon}{4} (\log t)^{-\beta/d}\right)+\Q(G_t(0))\\
 &\le 1-c_2(\varepsilon)t^{-d}\exp((\log t)^{\delta}) + \exp\left(-c_1(\varepsilon) (\log t)^{ d N+\beta (N-1/d)}\right)\\
 &\le 1-c_3(\varepsilon)t^{-d}\exp((\log t)^{\delta}). \end{align*}
 Hence,
  \begin{align*}\Q(\cap_{ z\in I_t}(F_t(z)\cup G_t(z))) &\le \Q(\cap_{ z\in I_t}F_t(z))+\Q(\cup_{ z\in I_t}  G_t(z) )\\
  &\le \left[1-c_3(\varepsilon)t^{-d}\exp((\log t)^{\delta})\right]^{c_4t^d(\log t)^{-(\kappa+N)d}}\\
  &\quad + \exp\left(-c_5(\varepsilon) (\log t)^{ d N+\beta (N-1/d)}\right)\\
 &\le \exp\left(-c_6(\varepsilon)\exp((\log t)^{\delta})(\log t)^{-(\kappa+N)d}\right)\\
 &\quad+\exp\left(-c_5(\varepsilon) (\log t)^{ d N+\beta (N-1/d)}\right)\\
 &\le \exp\left(-c_7(\varepsilon) (\log t)^{ d N+\beta (N-1/d)}\right),\end{align*} where in the third inequality we used the fact that $1-x\le e^{-x}$ for all $x>0$.
 The Borel-Cantelli lemma yields that $\Q$-almost surely for all $t$ large enough there exists $z:=z(t,\w)\in I_t$ for which both $F_t(z)$ and $G_t(z)$ fail to happen.

Below, we will fix this $z\in I_t$ for all $t$ large enough. Then, it holds that
\begin{align*} \lambda_{V^\w, B(z,B_M(t))}&\le \lambda_1(B(z,B_M(t)))+\sup_{x\in B(z,B_M(t))}V^\w(x)\\
&\le  \lambda_1(B(0,B_M(t)))+\sup_{x\in B(z,B_M(t))}\tilde V^\w(x)+\sup_{x\in B(z,B_M(t))}\sum_{\w_i\notin z+\Lambda_N(t)}|z-\w_i|^{-d-\beta}\\
&\le 2M^{-\alpha}(\log t)^{-\beta/d}\lambda_1^{(\alpha)}(B(0,1)) + \lambda(t) +\frac{3\varepsilon}{4}(\log t)^{-\beta/d}, \end{align*} where in the last inequality we used Lemma \ref{L:lem-l33}.
Letting $\varepsilon$ small enough and $M$ large enough in the inequality above, we then prove the desired assertion.
 \end{proof}

\subsection{Quenched estimates of $u^{\omega}(t,0)$: critical case} In this part, we will briefly show that the arguments of Theorems \ref{T:upper} and \ref{T-low} with some modifications still work for the following
\begin{itemize}\item {\it critical case}\,\,{\rm (C)}:\,\, The characteristic exponent $\psi(\xi)$ of the pure-jump symmetric L\'evy process $Z$ fulfills that $\psi(\xi)=O(|\xi|^\alpha)$ as $|\xi|\to0$, and the shape function $\varphi$ in the random potential $V^\w(x)$  satisfies
\begin{equation}\label{cre}0<\liminf_{|x|\to\infty} \varphi(x) |x|^{d+\alpha}\le \limsup_{|x|\to\infty} \varphi(x) |x|^{d+\alpha}<\infty.\end{equation}\end{itemize}

\ \

In the critical case,
it was shown in \cite[Theorem 6.4]{Ok1} that the integrated density  $N (\lambda)$ of states of the random Schr\"{o}dinger operator $H$ defined by \eqref{e:density} satisfies that
$$-\infty<\liminf_{\lambda \to 0} \lambda^{d/\alpha} \log N (\lambda)\le \limsup_{\lambda \to 0} \lambda^{d/\alpha} \log N (\lambda) <0.$$
Then, according to the arguments in Subsection \ref{section3.1} and the proof of Theorem \ref{T:upper}, we have
\begin{theorem}\label{T:upper-c} In the critical case {\rm(C)}, assume that \eqref{e:upp} holds. Then, there is a constant $\kappa_0>0$ such that for any $\varepsilon>0$, $\Q$-almost surely there is $R_\varepsilon(\w)\ge1$ so that for any $R\ge \max\{R_\varepsilon(\w),\phi(t)\}$ and $t\ge 1$,
$$u^{\w}(t,0)\le  \Phi(t,R)+ C(\varepsilon) R^{d/2} \exp\left(-t(1-2\varepsilon)\left(\frac{ k_0}{d\log R}\right)^{\alpha/d}\right),$$ where $C(\varepsilon)$ is independent of $R$ and $t$. \end{theorem}
When $Z$ is a symmetric $\alpha$-stable process with the exponent $\psi^{(\alpha)}(\xi)$ given in \eqref{alpha} for some $\alpha\in (0,2]$, and $K:=\lim_{|x|\to\infty}{\varphi(x)}{|x|^{d+\alpha}}\in (0,\infty)$, \^{O}kura proved the precise annealed asymptotics of $u^\w(t,x)$ in \cite[Theorem and Remark ii]{Ok2};  that is,  for all $x\in \R^d$,
$$\lim_{t\to\infty} \frac{\log \Ee_\Q[u^{\omega}(t,x)]}{t^{d/(d+\alpha)}}=- C(\rho,K),$$
where $$C(\rho,K)=\inf_{f\in L^2(\R^d;dx)\cap B_c(\R^d): \|f\|_{L^2(\R^d;dx)}=1}\left\{D(f,f)+W(f^2)\right\}$$ with $B_c(\R^d)$ being the set of measurable functions with compact support,
$D(f,f)$ being the Dirichlet form associated with the symmetric $\alpha$-stable process $Z$, and $$W(f^2)=\rho\int_{\R^d} \left[1-\exp\left(-K\int_{\R^d}\frac{f^2(z)}{|x-z|^{d+\alpha}}\,dz\right)\right]\,dx.$$
Then, by the Tauberian theorem of exponential type (see \cite[Theorem 3]{Kas}), we have
$$\lim_{\lambda \to0}\lambda^{d/\alpha} \log N (\lambda)=- k_0:=- \frac{\alpha}{d+\alpha}\left(\frac{d}{d+\alpha}\right)^{d/\alpha} C(\rho,K).$$ So, in this case we have a precise expression for the constant $\kappa_0$ in Theorem \ref{T:upper-c}.

 \ \

For quenched lower bounds of $u^\w(t,x)$, we have the following statement.
 \begin{theorem}\label{T-low-c} In the critical case {\rm(C)}, assume that \eqref{e:low} holds. Then, there is a constant $C_0>0$ such that for  any $\delta\in (0,1/2)$, $\kappa>1$, $a>1$, $\eta,\varsigma\in (0,1)$, $\Q$-almost surely  there is $R_{\kappa,a,\eta,\varsigma}(\w)\ge1$ so that for any $R\ge R_{\kappa,a,\eta,\varsigma}(\w)$ and $t\ge 1$,
\begin{align*}u^{\w}(t,0)\ge & C(\kappa,\delta,\eta,a)  M_{\kappa,\eta}(R)^{-4\delta d}[\Psi_\delta(2 M_{\kappa,\eta}(R))]^a  \exp\left(-a^2(1+\varsigma)C_0tR^{-\alpha} \right),\end{align*} where  \begin{equation}\label{e:ssskkk}M_{\kappa,\eta}(R)=R^{-\kappa}\exp\left(\frac{w_d\rho}{d}((1+2\eta)R)^d\right).\end{equation}
    \end{theorem}

To prove Theorem \ref{T-low-c}, we need the following proposition, which is analogous to Proposition \ref{L:lem-l3}.

\begin{proposition}\label{L:lem-l4} In the critical case {\rm(C)}, for any $\kappa>1$ and $\eta,\varsigma\in (0,1)$, $\Q$-almost surely there exists $r_{\kappa,\eta,\varsigma}(\w)>0$ so that for all $r\ge r_{\kappa,\eta,\varsigma}(\w)$, there is $z:=z(r,\w)\in \R^d$ with $|z|\le M_{\kappa,\eta}(r)$,
$$\lambda_{V^\w,B(z,r)}\le \left((1+\varsigma)\lambda_1^{(\alpha)}(B(0,1))+C_1(\eta)\right) r^{-\alpha},$$ where
 $M_{\kappa,\eta}(r)$ is defined by \eqref{e:ssskkk},
  $C_1(\eta)$ is a positive constant depending on $\eta$ only, and $\lambda_1^{(\alpha)}(B(0,1))$ is the principle Dirichlet eigenvalue for the symmetric $\alpha$-stable process with the exponent $\psi^{(\alpha)}(\xi)$ given in \eqref{e:a1-1} and killed upon exiting $B(0,1)$. \end{proposition}
\begin{proof} We use some notations from the proof of Proposition \ref{L:lem-l3}(i). For any $\kappa>1$ and $\eta\in (0,1)$, let
$I_r:=( (2(1+\eta)r)\Z^d)\cap \{z\in \R^d: |z|\le M_{\kappa,\eta}(r)\}$ for any $r>0$. Still define $\varphi_0(x)=\varphi_0(|x|)$ for any $x\in \R^d$, where $\varphi_0(r)=\sup_{|x|\ge r}\varphi(x)$ for $r\in [0,\infty)$. It is clear that $\varphi(x)\le \varphi_0(x)$, and $\varphi_0(r)$ is a  decreasing function on $[0,\infty)$ which satisfies that there are constants $c_1,c_2>0$ so that for $r$ large enough
\begin{equation}\label{label}c_1r^{-d-\alpha}\le \varphi_0(r)\le c_2r^{-d-\alpha},\end{equation} thanks to \eqref{cre}.

Now, for any $z\in I_r$, define $F_r(z)$  as in the proof of Proposition \ref{L:lem-l3}(i), and
$$G_r(z)= \left\{\sup_{y\in B(z,r)}\sum_{\omega_i\notin B(z,(1+\eta)r)}\varphi_0(y-\omega_i)\ge C_*   r^{-\alpha} \right\}$$ for some constant $C_*>0$ which is chosen later.
As shown in the proof of Proposition \ref{L:lem-l3}(i), for $r>1$ large enough,
$$\Q(\cap_{ z\in I_r} F_r(z))\le  \exp(-r^d/2).$$

On the other hand, by the deceasing property of $\varphi_0(r)$ and \eqref{label}, for all $r$ large enough,
\begin{align*} &\Q\left[\exp\left(\frac{1}{\varphi_0(\eta r)}\sup_{y\in B(0,r)}\sum_{\omega_i\notin B(0,(1+\eta)r)}\varphi_0(y-\omega_i) \right)\right]\\
&\le \Q\left[\exp\left(\frac{1}{\varphi_0(\eta r)} \sum_{\omega_i\notin B(0,(1+\eta)r)}\varphi_0(\eta |\omega_i|/(1+\eta)) \right)\right]\\
&= \exp\left(\rho\int_{\R^d\backslash B(0,(1+\eta)r)}\left(e^{\varphi_0(\eta r)^{-1}\varphi_0(\eta|z|/(1+\eta))}-1\right)\,dz   \right)\\
&\le \exp \left(e\rho(1+\eta)^d\int_{\R^d\backslash B(0,r)} \frac{\varphi_0(\eta|z|)}{\varphi_0(\eta r)}\,dz   \right)\\
&\le \exp\left( c_3 r^d\right),\end{align*} where $c_3:=c_3(\eta)$ is independent of $r$.
This along with the Markov inequality and \eqref{label} yields that for $r$ large enough
\begin{align*} \Q(G_r(0))\le & \Q\left[\exp\left(\frac{1}{\varphi_0(\eta r)}\sup_{y\in B(0,r)}\sum_{\omega_i\notin B(0,(1+\eta)r)}\varphi_0(y-\omega_i) \right)\ge \exp\left( \frac{C_* r^{-\alpha} }{\varphi_0(\eta r)}\right)\right]\\
\le&  \exp\left( c_3r^d- C_*c_2^{-1}\eta^{d} r^d\right).\end{align*}
Since $\{G_r(z)\}_{z\in I_r}$ have the same distribution (but are not independent with each other), we find that
\begin{align*} \Q(\cup_{ z\in I_r}G_r(z))
&\le   2\left(\frac{M_{\kappa,\eta}(r)}{(1+\eta)r}\right)^d \exp\left( c_3r^d- C_*c_2^{-1}\eta^{d} r^d\right)\\
&=  2(1+\eta)^{-d} r^{-(\kappa+1)d}\exp\left[(\rho w_d (1+2\eta)^d  +c_3 - C_*c_2^{-1}\eta^{d}) r^d\right].\end{align*}
Now, we take $$C_*=2c_2\eta^{-d}(\rho w_d (1+2\eta)^d  +c_3), $$ and so $$\Q(\cup_{ z\in I_r}G_r(z))\le 2(1+\eta)^{-d} r^{-(\kappa+1)d}\exp\left(- \frac{1}{2c_2}C_*\eta^{d} r^d\right).$$

Therefore, for all $r$ large enough, \begin{align*}\Q(\cap_{ z\in I_r}(F_r(z)\cup G_r(z))
 \le &\Q(\cap_{ z\in I_r}F_r(z))+\Q(\cup_{ z\in I_r}G_r(z))\\
 \le &\exp(-r^d/2)+2(1+\eta)^{-d} r^{-(\kappa+1)d}\exp\left(- \frac{1}{2c_2}C_*\eta^{d} r^d\right).\end{align*}
Hence, by the Borel-Cantelli lemma, $\Q$-almost surely there exists  $z:=z(r,\w)\in\R^d$ such that $|z|\le M_{\kappa,\eta}(r)$,  and both $F_r(z)$ and $G_r(z)$  fail to hold.

With this at hand, one can follow the proof of  Proposition \ref{L:lem-l3}(i) to get the desired assertion. \end{proof}

According to Proposition \ref{L:lem-l4}, one can repeat the argument for  Theorem \ref{T-low} to prove Theorem \ref{T-low-c}.  Furthermore, as an application of Theorems \ref{T:upper-c} and \ref{L:lem-l4}, we also can obtain Proposition \ref{P:cre}. The details are omitted here.

\bigskip

\noindent{\bf Acknowledgement.} The research  is supported by the National Natural Science Foundation of China (No.\ 11831014), the Program for Probability and Statistics: Theory and Application (No.\ IRTL1704), and the Program for Innovative Research Team in Science and Technology in Fujian Province University (IRTSTFJ).

\end{document}